\documentclass[AMA,Times1COL]{WileyNJDv5} 

\articletype{Article Type}%

\received{Date Month Year}
\revised{\today}
\accepted{Date Month Year}
\journal{Numerical Methods for Partial Differential Equations}
\volume{00}
\copyyear{2023}
\startpage{1}

\usepackage{mathbbol}
\DeclareSymbolFontAlphabet{\amsmathbb}{AMSb}%
\usepackage{tikz}
\usetikzlibrary{decorations.pathmorphing,patterns}
 
\DeclareSymbolFontAlphabet{\amsmathbb}{AMSb}

\definecolor{lightblue}{rgb}{0.22,0.45,0.70}
\definecolor{lightgreen}{rgb}{0.22,0.50,0.25}

\newcommand\cero{\boldsymbol{0}}
\newcommand{\norm}[1]{\left\|#1\right\|}
\newcommand{\jump}[1]{\left[\!\left[#1\right]\!\right]}
\newcommand{\avg}[1]{\left\{\!\!\left\{#1\right\}\!\!\right\}}
\newcommand{\fracd}[2]{\displaystyle
	{\frac{{\displaystyle{#1}}}{{\displaystyle{#2}}}}}
\newcommand\vdiv{\mathop{\mathrm{div}}\nolimits}
\newcommand\curl{\mathop{\mathrm{curl}}\nolimits}
\newcommand\bcurl{\mathop{\mathbf{curl}}\nolimits}
\newcommand\bdiv{\mathop{\mathbf{div}}\nolimits}

\newcommand{\eps}{\varepsilon}
\newcommand{\beps}{\boldsymbol{\varepsilon}}

\newcommand{\bS}{\mathbf{S}}
\newcommand\bU{\mathbf{U}}
\newcommand\bM{\mathbf{M}}
\newcommand\bW{\mathbf{W}}

\newcommand{\bI}{\mathbf{I}}

\newcommand\bb{\boldsymbol{b}}
\newcommand{\bu}{\boldsymbol{u}}

\newcommand{\bn}{\boldsymbol{n}}

\newcommand{\bv}{\boldsymbol{v}}
\newcommand{\bw}{\boldsymbol{w}}
\newcommand{\bx}{\boldsymbol{x}}
\newcommand{\by}{\boldsymbol{y}}

\newcommand{\bz}{\boldsymbol{z}}

\newcommand{\ff}{\boldsymbol{f}}

\newcommand{\bnabla}{\boldsymbol{\nabla}}

\newcommand\cA{\mathcal{A}}
\newcommand\cB{\mathcal{B}}

\newcommand\cT{\mathcal{T}}
\newcommand\cE{\mathcal{E}}
\newcommand\cF{\mathcal{F}}

\newcommand\cL{\mathcal{L}}

\newcommand\bP{\mathbf{P}}

\newcommand\bV{\mathbf{V}}

\newcommand\PP{\mathrm{P}}

\newcommand\bH{\mathbf{H}}
\newcommand\bL{\mathbf{L}}
\newcommand\rL{\mathrm{L}}

\newcommand\rZ{\mathrm{Z}}

\newcommand\bX{\mathbf{X}}

\newcommand\rH{\mathrm{H}}
\newcommand\rQ{\mathrm{Q}}

\newcommand\rU{\mathrm{U}}

\newcommand\bgamma{\boldsymbol{\gamma}}

\newcommand\bDelta{\boldsymbol{\Delta}}

\newcommand\bomega{\boldsymbol{\omega}}
\newcommand\btheta{\boldsymbol{\theta}}
\newcommand\bzeta{\boldsymbol{\zeta}}

\numberwithin{equation}{section}
\numberwithin{figure}{section}
\numberwithin{table}{section}
\numberwithin{theorem}{section}
\numberwithin{lemma}{section}
\numberwithin{remark}{section}
\allowdisplaybreaks

\begin{document}

\title{Robust finite element methods and solvers for the Biot--Brinkman equations in vorticity form}

\author[1]{Ruben Caraballo}

\author[2]{Chansophea Wathanak In}

\author[3]{Alberto F. Mart\'in}

\author[2,4]{Ricardo Ruiz-Baier}

\authormark{CARABALLO \textsc{et al.}}
\titlemark{Biot--Brinkman equations in vorticity form}

\address[1]{\orgdiv{GIMNAP, Departamento de Matem\'atica}, \orgname{Universidad del B\'io-B\'io}, \orgaddress{\state{Concepci\'on}, \country{Chile}}}

\address[2]{\orgdiv{School of Mathematics}, \orgname{Monash University}, \orgaddress{\state{9 Rainforest Walk, Melbourne VIC 3800}, \country{Australia}}}

\address[3]{\orgdiv{School of Computing}, \orgname{Australian National University}, \orgaddress{\state{Acton ACT 2601}, \country{Australia}}}

\address[4]{\orgname{Universidad Adventista de Chile}, \orgaddress{\state{Casilla 7-D, Chill\'an}, \country{Chile}}}

\corres{Corresponding author Ricardo Ruiz-Baier. \email{ricardo.ruizbaier@monash.edu}}


\fundingInfo{Chilean Research and Technology Council through the ANID program for international research visits; by the Australian Government through the National Computational Infrastructure (NCI) under the ANU Merit Allocation Scheme (ANUMAS); and by the Australian Research Council through the Future Fellowship grant FT220100496 and Discovery Project grant DP22010316.}

\abstract[Abstract]{In this paper, we propose a new formulation and a suitable finite element method for the steady coupling of viscous flow in deformable porous media using divergence-conforming filtration fluxes. The proposed method is based on the use of parameter-weighted spaces, which allows for a more accurate and robust analysis of the continuous and discrete problems. Furthermore, we conduct a solvability analysis of the proposed method and derive optimal error estimates in appropriate norms. These error estimates are shown to be robust in {a variety of regimes, including} the case of large Lam\'e parameters and small permeability and storativity coefficients. To illustrate the effectiveness of the proposed method, we provide a few representative numerical examples, including convergence verification {and testing of} 
robustness of block-diagonal preconditioners with respect to model parameters.}

\keywords{Biot--Brinkman coupled problem, deformable porous media, vorticity-based formulation, mixed finite element methods. \\[2ex]
}

\jnlcitation{\cname{%
\author{Caraballo R},
\author{In CW},
\author{Mart\'in AF}, and
\author{Ruiz-Baier R}}.
\ctitle{Robust finite element methods and solvers for the Biot--Brinkman equations in vorticity form.} \cjournal{\it Numer Methods PDEs.} \cvol{2023;00(00):1--18}.}

\maketitle

\renewcommand\thefootnote{}

\renewcommand\thefootnote{\fnsymbol{footnote}}
\setcounter{footnote}{1}

\section{Introduction}

We address the analysis of the Biot--Brinkman equations, which serve as a model for filtration of viscous flow in deformable porous media \cite{ambar15,carrillo21,rajagopal07,rohan19}. The system has been recently analysed in \cite{hong22} for the case of multiple network poroelasticity, using $\bH(\vdiv,\Omega)$-conforming displacements and filtration fluxes (or seepage velocities) for each compartment, also designing robust preconditioners. Here we propose a reformulation for only one fluid compartment but using the vorticity field (defined as the curl of the filtration velocity) as an additional unknown in the system, and we also include the total pressure, following \cite{oyarzua16,lee17}. Such an approach enables us to avoid the notorious problem of locking or non-physical pressure oscillations when approximating poroelastic models and it has led to a number of developments including extensions to multiple network models, interfacial free-flow and poromechanics coupling, nonlinear interaction with species transport, reformulations into four and more fields systems, and using other discretisations such as discontinuous Galerkin, nonconforming FEM, weak Galerkin, and virtual elements. See, e.g.,  \cite{boon21,buerger21,hong19,ju20,kumar20,lee19,piersanti20,qi20,ruiz22,verma22,zhang20}. 

The formulation of viscous flow equations using vorticity, velocity and pressure has been used and analysed (in terms of solvability of the continuous and discrete formulations and deriving error estimates) extensively in, e.g.,  \cite{amara07,amara04,anaya23,anaya15,anaya21,anaya16,bernardi06,chang90,duan03,dubois03,salaun15}. Methods based on vorticity formulations are useful for visualisation of rotational flows and they are convenient when dealing with rotation-based boundary conditions. The coupling with other effects such as mass and energy transport has also been addressed, see for example \cite{anaya18,bernardi18,lenarda17}. These contributions include fully mixed finite elements, augmented forms, spectral methods, and Galerkin least-squares stabilised types of discretisations. At the continuous level, one appealing property of some of these vorticity formulations is that   {the $\bH^1$-conformity of the filtration flux is relaxed} and velocity is sought in $\bH(\vdiv,\Omega)$ and the vorticity is sought in either $\bH(\bcurl,\Omega)$ or $\bL^2(\Omega)$. {Then, using simply a \emph{conforming} method, resulting discretisations are readily mass conservative. Also, mixed methods that look for both vorticity and velocity will typically deliver approximate vorticity \emph{with the same accuracy} as velocity (as opposed to schemes needing numerical differentiation to get vorticity from approximate velocity). Another advantage of using a vorticity-based method is that an exactness sequence exists between the spaces for velocity, vorticity, and pressure, yielding a framework that can be straightforwardly analysed with finite element exterior calculus \cite{hanot23}. We also stress that solving directly for vorticity enables a more thorough investigation of vortical patterns, interactions at boundary layers, and the formation of turbulent eddies, all of which are crucial in a diverse array of applications.}

{Note however that, differently from the aforementioned} works, in the case of the Biot--Brinkman problem, the divergence of the fluid velocity is not zero (or a prescribed fluid source), but it depends on the velocity of the solid and on the rate of change of fluid pressure. Also, an additional term of grad--div type appears in the momentum balance for the fluid.  

In this paper, we prove the well-posedness of the continuous and discrete formulations for the coupling of mechanics and fluid flow in fluid-saturated deformable porous media using Banach--Ne\v{c}as--Babu\v{s}ka theory in parameter-weighted {Hilbert} spaces. The appropriate choice of weighting parameters yields automatically a framework for robust operator preconditioning in the Biot equations, following the approach from \cite{lee17}. This operator scaling yields robustness with respect to the elastic parameters, storativity, Biot--Willis coefficient, and with respect to permeability. For the Brinkman component,  our present formulation is such that the filtration flux terms have a different weight in their $\bL^2(\Omega)$ and $\bH(\vdiv,\Omega)$ contributions, which  requires a different treatment for the analysis of the pore pressure terms. To address this issue it suffices to appeal to  the recent theory in \cite{baerland20} (see also \cite{mardal11}), which was developed for Darcy equations using non-standard sum spaces, and we appropriately modify the scalings in the momentum equation. This modification entails the use of (discontinuous) Laplacian operators in the fluid pressure preconditioning.

Our proposed approach also offers a novel contribution to the field of operator preconditioning for the interaction of mechanics and fluid flow in fluid-saturated deformable porous media, which are challenging to solve, and the design of efficient preconditioners is highly problem-dependent \cite{efendiev09}. Previous works have explored the use of block-diagonal preconditioners, Schur complements, and pressure-correction methods, which have improved the convergence rate and computational efficiency of numerical solutions for poromechanics problems \cite{chen20,giraud11, jha15,liu20}. In this work, we also derive parameter-robust solvers, but following  \cite{hong22} and also \cite{ju20,lee19,piersanti20}. Our results confirm that, additionally to the   Laplacian contribution  needed in the Riesz preconditioner associated with the fluid pressure mentioned above, we also require off-diagonal contributions in the total pressure and fluid pressure coupling terms (as employed in \cite{boon21}). The overall parameter scalings that we propose are motivated by the stability analysis and we verify computationally that robustness  holds for this particular choice. Note that we only discuss one type of boundary conditions, but the extension to other forms can be adapted accordingly. 
 
Research on preconditioning techniques for advanced discretisations of block multiphysics systems has also crystallised in a number of high quality open source software packages. One of the earliest efforts was the BKPIT C++ package \cite{chow98} which, following an object-oriented approach, provides an extensible framework for the implementation of {\em algebraic} block preconditioners, such as block Jacobi or block Gauss--Seidel. More recently, as the field of physics-based and discretisation tailored preconditioners has evolved with breakthrough inventions in approximate block factorisation and Schur-complement methods towards ever faster and scalable iterative solvers for large-scale systems, more sophisticated block preconditioning software is available. The widely-used PETSc package offers the PCFIELDSPLIT subsystem \cite{brown12} to design and compose complex block preconditioners. The Firedrike library extends PETSc block preconditioning capabilities and algebraic composability further~\cite{kirby18}. Another tightly-related effort is the Teko Trilinos package \cite{cyr16}, which provides a high-level interface to compose block preconditioners using a functional programming style in C++. The authors in \cite{badia14} present a generic software framework in object-oriented Fortran to build block recursive algebraic factorisation preconditioners for double saddle-point systems, as those arising in MagnetoHydroDynamics (MHD). 

In this paper, we build upon the high momentum gained in the last years by the Julia programming language for scientific and numerical computing. In particular, the realisation of the numerical discretisation and preconditioning algorithms is conducted with the \texttt{Gridap} finite element software package \cite{badia20}. We leverage the flexibility of this framework, and its composability with others in the Julia package ecosystem, such as \texttt{LinearOperators} \cite{orban20}, to prototype natural Riesz map preconditioners in the sum spaces described above, leading to complex multiphysics coupling solvers that are robust with respect to physical parameters variations and mesh resolution. For the sake of reproducibility, the Julia software used in this paper is available publicly/openly at \cite{caraballo23}. 

The remainder of this article is organised as follows. The presentation of the new form of Biot--Brinkman equations and its weak formulation are given in Section \ref{sec:model}. The modification of the functional structure to include parameter weights and the unique solvability analysis for the continuous problem are addressed in Section~\ref{sec:wellp}. The definition of the finite element discretisation and the specification of the well-posedness theory for the discrete problem is carried out in Section \ref{sec:FE}. The error analysis (tailored for a specific family of finite elements but applicable to other combinations of discrete spaces as well) is detailed also in that section. Numerical experiments are collected in Section~\ref{sec:results} and we close in Section~\ref{sec:concl} with a brief summary and a discussion on possible extensions.

\section{Model problem and its weak formulation}\label{sec:model}
\subsection{Preliminaries}
Let us consider a simply connected bounded and Lipschitz domain $\Omega\subset\mathbb{R}^d$, $d \in \{2,3\}$ occupied by a poroelastic domain with one incompressible fluid network incorporating viscosity. 
The domain boundary is denoted as $\Gamma:=\partial\Omega$.  
Throughout the text, given a normed space $S$, by $\bS$ and $\mathbb{S}$ we will denote the vector and tensor extensions, $S^d$ and $S^{d \times d}$, respectively. In addition, by $\rL^2(\Omega)$ we will denote 
the usual Lebesgue space of square integrable functions and $\rH^m(\Omega)$ denotes the usual Sobolev space with weak derivatives of order up to $m \geq 0$ in $\rL^2(\Omega)$, and use the convention that $\rH^0(\Omega) = \rL^2(\Omega)$. 

Next, we recall the definition of the following Hilbert spaces 
\begin{gather*}\bH^{1+m}(\Omega) := \{\bgamma\in \bH^m(\Omega): \bnabla\bgamma \in \mathbb{H}^m(\Omega)\}, \\
  \bH^m(\vdiv,\Omega)\!:=\!\{\bzeta \in \bH^m(\Omega)\!: \vdiv\bzeta \in \rH^m(\Omega) \}, \quad   
  \bH^m(\bcurl,\Omega)\!:=\!\{\btheta \in \bH^m(\Omega)\!: \bcurl\btheta \in \bH^m(\Omega) \},
  \end{gather*}
and in the case that $m=0$ the latter two spaces are denoted $\bH(\vdiv,\Omega)$ and $\bH(\bcurl,\Omega)$, and 
we use the following notation
for the typical norms associated with such spaces 
\begin{gather*} \norm{\bu}_{1,\Omega}^2:=\norm{\bu}^2_{0,\Omega}+\norm{\bnabla\bu}^2_{0,\Omega}, \qquad    
\norm{\bv}_{\vdiv,\Omega}^2:=\norm{\bv}^2_{0,\Omega}+\norm{\vdiv\bv}^2_{0,\Omega}, \qquad   \norm{\bzeta}_{\bcurl,\Omega}^2:=\norm{\bzeta}^2_{0,\Omega}+\norm{\bcurl\bzeta}^2_{0,\Omega},\end{gather*}
respectively. Furthermore,  in view of the boundary conditions on $\Gamma$ (to be made precise below), we also use the following notation for relevant subspaces 
\begin{align}\label{eq:Hstar}
\bH^1_0(\Omega) &:=\{\bgamma\in \bH^1(\Omega): \bgamma = \cero \text{ on } \Gamma\}, \nonumber \\
  \bH_0(\vdiv,\Omega) & := \{\bzeta\in \bH(\vdiv,\Omega): \bzeta\cdot\bn = 0 \text{ on } \Gamma\}, \\
  \bH_0(\bcurl,\Omega)& := \{\btheta\in \bH(\bcurl,\Omega): \btheta\times\bn = \cero \text{ on } \Gamma\}, \nonumber\\
\rL^2_0(\Omega) &:=\{\zeta \in \rL^2(\Omega): \int_\Omega \zeta = 0\}.
  \nonumber\end{align}

For a generic functional space $X$ and a scalar $\eta>0$, the weighted space $\eta X$ refers to the same $X$ but endowed with 
the norm $\eta \|\cdot\|_X$. {In addition, we recall the definition of the norm of intersection $X\cap Y$ and sum $X+Y$ of Hilbert spaces $X,Y$} 
\begin{equation}\label{eq:sum}
  \|z \|^2_{X \cap Y} = \|z\|^2_X +\|z\|^2_Y, \qquad  \|z \|^2_{X + Y} = \inf_{\small\begin{array}{c}z=x+y\\x\in X,y\in Y\end{array}}  \|x\|^2_X +\|y\|^2_Y.\end{equation}

\subsection{The governing equations}
The viscous filtration flow through the deformed porous skeleton can be described by the following form of the Biot--Brinkman equations in steady form, representing the mixture momentum, fluid momentum, and mixture mass balance, respectively
\begin{subequations}
\begin{align}\label{eq:biot-brinkman0}
-\bdiv(2\mu\beps(\bu) + [\lambda\vdiv\bu -\alpha p]\bI) & = \bb,\\
{\frac{1}{\kappa}}\bv - {\frac{\nu}{\kappa}} \bDelta \bv + {\nabla p}  & = \widehat{\ff}, \label{eq:biot-brinkman1}\\
- c_0 p - \alpha \vdiv \bu - \vdiv \bv & = g,  
\end{align}\end{subequations}
where $\bu$ is the displacement of the skeleton, $\beps(\bu) = \frac{1}{2} (\bnabla\bu + \bnabla\bu^{\tt t})$ is the tensor of infinitesimal strains, $\bv$ is the filtration flux, and $p$ is the pressure head. The parameters are the external body load $\bb$, the external force applied on the fluid $\widehat{\ff}$, the kinematic viscosity of the interstitial fluid $\nu$, the hydraulic conductance (permeability field, here assumed a positive constant) $\kappa$, the Lam\'e coefficients of the solid structure $\lambda,\mu$, the storativity $c_0$, and the Biot--Willis modulus $\alpha$. Equations \eqref{eq:biot-brinkman0} are equipped with
boundary conditions of clamped boundary for the solid phase and slip filtration velocity
\[
\bu = \cero \quad \text{and}\quad \bv\cdot\bn = 0 \quad \text{on}\quad \Gamma.
\]

Next we introduce the rescaled filtration vorticity vector
\begin{equation}\label{eq:vorticity}
  \bomega := {\sqrt{\frac{\nu}{\kappa}}} \bcurl \bv,\end{equation}
which has a different weight than that used in \cite{anaya16} for Brinkman flows. We also define, following the developments in \cite{oyarzua16,lee19}, the additional total pressure field 
\[ \varphi:= - \lambda\vdiv\bu + \alpha p.\]
In order to rewrite the fluid momentum balance in terms of the rescaled filtration vorticity \eqref{eq:vorticity}, we employ the following vector identity, valid for a generic vector field $\bv$: 
\[ \bcurl \bcurl \bv = -\boldsymbol{\Delta}\bv + \nabla(\vdiv \bv).\]
These steps, together with a rescaling of the external force $\ff = {\nu}\widehat{\ff}$, lead to the following equations (mixture momentum, constitutive equation for total pressure, fluid momentum, constitutive equation for filtration vorticity, and mixture mass balance)
\begin{subequations}
\begin{align}\label{eq:biot-brinkman}
-\bdiv(2\mu\beps(\bu) - \varphi\bI) & = \bb,\\
-\frac{1}{\lambda}\varphi + \frac{\alpha}{\lambda} p - \vdiv\bu & = 0,\label{eq:biot-brinkman2} \\
{\frac{1}{\kappa}} \bv + {\sqrt{\frac{\nu}{\kappa}}} \bcurl\bomega - {\frac{\nu}{\kappa}}\nabla(\vdiv\bv) + \nabla p  & = \ff, \\
- \bomega +{\sqrt{\frac{\nu}{\kappa}}} \bcurl \bv & = \cero, \\
- (c_0 + \frac{\alpha^2}{\lambda}) p + \frac{\alpha}{\lambda}\varphi - \vdiv \bv & = g;\label{eq:biot-brinkman4}
\end{align}
and the boundary conditions 
now read 
\begin{equation}\label{bc:gamma}
  \bu = \cero \quad \text{and}\quad \bv\cdot\bn = 0 \quad \text{and}\quad \bomega \times\bn = \cero \qquad \text{on}\quad \Gamma.
\end{equation}
\end{subequations}
\subsection{Weak formulation}

Let us assume that $\bb,\ff \in \bL^2(\Omega)$, $g \in \rL^2(\Omega)$
and that all other model coefficients are positive constants. 
We proceed to multiply the governing equations by suitable test functions and to integrate by parts over the domain.
Note that for the divergence-based terms we appeal to the usual form of the Gauss formula, whereas for curl-based terms we use the following result from, e.g., \cite{girault}: 
\[ \int_\Omega \bcurl\bgamma\cdot \btheta = \int_\Omega \bgamma\cdot \bcurl\btheta + \int_{\partial\Omega} (\btheta\times\bn)\cdot \bgamma \qquad \forall \btheta\in\bH(\bcurl,\Omega),\ \bgamma\in \bH^1(\Omega),\]
together with the scalar triple product identity
\[(\btheta\times\bn)\cdot \bv = \btheta\cdot(\bn\times\bv).\]

We observe in advance that equations \eqref{eq:biot-brinkman2},\eqref{eq:biot-brinkman4} suggest that, in the limit of $c_0\to 0, \lambda\to \infty$, both pressures are not uniquely defined and so we require to add a constraint on their mean value (to be zero, for example). We then arrive at the following weak formulation for \eqref{eq:biot-brinkman}: Find $(\bu,\bv,\bomega,\varphi,p)\in \bH^1_0(\Omega)\times\bH_0(\vdiv,\Omega)\times \bH_0(\bcurl,\Omega) \times \rL^2_0(\Omega) \times \rL^2_0(\Omega)$ 
such that 
\begin{subequations}\label{eq:weak} 
  \begin{align}
2\mu \int_\Omega  \beps(\bu): \beps(\bgamma) - \int_\Omega \varphi \vdiv \bgamma & = 
\int_\Omega \bb\cdot\bgamma \quad  \forall \bgamma\in\bH^1_0(\Omega), \\
{\frac{1}{\kappa}} \int_\Omega\!\bv\cdot\bzeta + {\sqrt{\frac{\nu}{\kappa}}}  \int_\Omega\!\!\bcurl\bomega\cdot\bzeta
+ {\frac{\nu}{\kappa} } \int_\Omega\!\vdiv\bv\,\vdiv\bzeta -\int_\Omega p\,\vdiv \bzeta & = \int_\Omega \ff\cdot\bzeta \quad  \forall \bzeta\in\bH_0(\vdiv,\Omega), \\
-\int_\Omega \bomega\cdot\btheta + {\sqrt{\frac{\nu}{\kappa}}} \int_\Omega\bcurl\btheta\cdot\bv & = 0 \quad \forall \btheta \in \bH_0(\bcurl,\Omega),\\
-\frac{1}{\lambda}\int_\Omega \varphi\,\psi + \frac{\alpha}{\lambda} \int_\Omega p\psi - \int_\Omega \psi\,\vdiv\bu & = 0 \quad \forall \psi \in \rL_0^2(\Omega),\\
- (c_0+ \frac{\alpha^2}{\lambda}) \int_\Omega p\,q+ \frac{\alpha}{\lambda} \int_\Omega \varphi\,q - \int_\Omega q\,\vdiv\bv & = \int_\Omega g\, q \quad \forall q\in \rL_0^2(\Omega),
\end{align}\end{subequations}
where we have also used the boundary conditions \eqref{bc:gamma}.

Let us now define the following bilinear forms and linear functionals
\begin{gather*}
a_1 (\bu,\bgamma) :=  2\mu \int_\Omega  \beps(\bu): \beps(\bgamma), \qquad
b_1(\bgamma,\psi) : = - \int_\Omega \psi\vdiv \bgamma, \qquad 
\hat{b}_1(\bzeta,q) : = - \int_\Omega q\vdiv \bzeta, \\
a_2(\bv,\bzeta):= {\frac{1}{\kappa}} \int_\Omega\bv\cdot\bzeta +{ \frac{\nu}{\kappa}}  \int_\Omega\vdiv\bv\,\vdiv\bzeta, \qquad  b_2(\btheta,\bzeta) : =   {\sqrt{\frac{\nu}{\kappa}}} \int_\Omega \bcurl\btheta \cdot \bzeta, \\
a_3(\bomega,\btheta):= \int_\Omega \bomega \cdot\btheta, \qquad 
a_4(\varphi,\psi):= \frac{1}{\lambda} \int_\Omega \varphi\,\psi, \qquad 
b_3(q,\psi):= \frac{\alpha}{\lambda} \int_\Omega q\,\psi, \qquad 
a_5(p,q):= \biggl(c_0+ \frac{\alpha^2}{\lambda}\biggr) \int_\Omega p\,q,
\\
B(\bgamma):=  \int_\Omega \bb\cdot\bgamma, \qquad
F(\bzeta):= \int_\Omega \ff\cdot\bzeta, \qquad G(q):= \int_\Omega g\,q, 
  \end{gather*}
with which \eqref{eq:weak} is rewritten as follows: find $(\bu,\bv,\bomega,\varphi,p)\in \bH^1_0(\Omega)\times\bH_0(\vdiv,\Omega)\times \bH_0(\bcurl,\Omega) \times \rL_0^2(\Omega) \times \rL_0^2(\Omega)$ 
such that
\begin{subequations}\label{eq:weak2}
\begin{alignat}{5}
&a_1(\bu,\bgamma) &                &                       &+\;b_1(\bgamma,\varphi)&                       &=&\;B(\bgamma)&\qquad\forall \bgamma\in\bH^1_0(\Omega), \\
&                 &a_2(\bv,\bzeta) &+\;b_2(\bomega,\bzeta) &                       &+\;\hat{b}_1(\bzeta,p) &=&\;F(\bzeta) &\qquad\forall\bzeta\in \bH_0(\vdiv,\Omega),\\
&                 &b_2(\btheta,\bv)&-\;a_3(\bomega,\btheta)&                       &                       &=&\;0         &\qquad\forall \btheta\in \bH_0(\bcurl,\Omega),\\
&b_1(\bu,\psi)    &                &                       &-\;a_4(\varphi,\psi)   &+\;b_3(p,\psi)         &=&\;0         &\qquad\forall \psi\in \rL_0^2(\Omega),\\
&                 &\hat{b}_1(\bv,q)&                       &+\;b_3(q,\varphi)      &-\;a_5(p,q)            &=&\; G (q)&\qquad \forall q\in \rL_0^2(\Omega).
\end{alignat}
\end{subequations}

\section{Solvability analysis}\label{sec:wellp}

\subsection{Preliminaries}

The well-posedness analysis for \eqref{eq:weak2} will be put in the framework of the abstract  Banach--Ne\v{c}as--Babu\v{s}ka theory, which we state next (see, e.g., \cite{ern04}). 
\begin{theorem}\label{BNB}
Let $(E_1, \|\cdot\|_{E_1})$ be a reflexive Banach space, $(E_2,\|\cdot\|_{E_2})$ a Banach space, and $T:E_1\rightarrow E_2'$ a bounded, linear form satisfying the followings conditions:
\begin{enumerate}
\item[] {\bf (BNB1)} For each $y\in E_2\setminus \{0\}$, there exists $x\in E_1$ such that
  \begin{equation}\label{BNB1}
    \langle T(x), y\rangle_{E_2',E_2}\neq 0.\end{equation} 
\item[] {\bf (BNB2)} There exists $c>0$ such that 
\begin{equation}\label{BNB2}
\|T(x) \|_{E_2'} \geq c\|x\|_{E_1} \text{ for all } x\in E_1. 
\end{equation}
Then, for every $x^{*}\in E_2'$ there exists a unique $x\in E_1$ such that
\[T(x)=x^{*}.\]
\end{enumerate}
\end{theorem}
\noindent 
Let us first consider the product space 
\begin{align*} \bX & := \bH^1_0(\Omega)\times \bH_0(\vdiv,\Omega) \times  \bH_0(\bcurl,\Omega) \times \rL_0^2(\Omega) \times \rL_0^2(\Omega),
\end{align*}
and, using the notation $\vec\bx:=(\bu,\bv,\bomega,\varphi,p)\in \bX $, we proceed to equip this space with the norm
\begin{align}\label{quasi-norm}
  \|\vec\bx\|_{\bX}^2 & := 2\mu \|\beps(\bu)\|^2_{0,\Omega}+ {\frac{1}{\kappa}}\|\bv\|^2_{0,\Omega} + {\frac{\nu}{\kappa}}\|\vdiv\bv\|^2_{0,\Omega} +\|\bomega\|^2_{0,\Omega} +\nu\|\bcurl\bomega\|^2_{0,\Omega}
  + {c_0 \|p\|^2_{0,\Omega} + \frac{1}{\lambda}\|\varphi + \alpha p\|^2_{0,\Omega}.}
\end{align}

{\begin{remark}\label{rem:scale0}
Note that the scaling of the vorticity seminorm with the viscosity (fifth term in the right-hand side of \eqref{quasi-norm} is not readily apparent from the weak formulation \eqref{eq:weak2}, but it is suggested by the stability analysis in Lemma~\ref{cont-quasicoerc}, below. 
\end{remark}}

Let us also introduce the bilinear form
\begin{align}\nonumber
\langle \mathcal{A}_{\epsilon}(\vec{\boldsymbol{x}}),\vec{\boldsymbol{y}}\rangle&:=a_1(\bu,\bgamma)+b_1(\bgamma,\varphi)+a_2(\bv,\bzeta)+b_2(\bomega,\bzeta)+\hat{b}_1(\bzeta,p)+b_2(\btheta,\bv)-a_3(\bomega,\btheta)\\
&\qquad+b_1(\bu,\psi)-a_4(\varphi,\psi)+b_3(p,\psi)+\hat{b}_1(\bv,q)+b_3(q,\varphi)-a_5(p,q), \label{bilinearform}
\end{align}
induced by the operator $\mathcal{A}_{\epsilon}:\bX\rightarrow \bX'$
(where the subscript $\epsilon$ indicates dependence with respect to the model parameters $\kappa,\alpha,\mu,\nu,c_0,\lambda$), and again we emphasise that  we have different scalings than those used in \cite{anaya19,hong22}.

From the Cauchy--Schwarz inequality we readily have the following bounds for the bilinear forms in \eqref{eq:weak2}
\begin{gather}\label{aux-bounds}
a_1(\bu,\bgamma)\leq 2\mu |\bu|_{1,\Omega}|\bgamma|_{1,\Omega},\qquad b_1(\bgamma,\varphi) \leq \Vert \varphi \Vert_{0,\Omega} |\bgamma|_{1,\Omega},\qquad \hat{b}_1(\bzeta,q)\leq \Vert q\Vert_{0,\Omega} \Vert \vdiv \bzeta\Vert_{0,\Omega},\nonumber\\
 a_2(\bv,\bzeta)\leq \dfrac{1}{\kappa}\Vert\bv \Vert_{0,\Omega}\Vert \bzeta\Vert_{0,\Omega}+\dfrac{\nu}{\kappa}\Vert \vdiv \bv\Vert_{0,\Omega}\Vert \vdiv \bzeta \Vert_{0,\Omega},\qquad 
 b_2(\btheta,\bzeta)\leq \sqrt{\frac{\nu}{\kappa}}\Vert \curl \btheta\Vert_{0,\Omega}\Vert\bzeta\Vert_{0,\Omega},\nonumber\\
a_3(\bomega,\btheta)\leq \Vert\bomega \Vert_{0,\Omega} \Vert \btheta\Vert_{0,\Omega},\qquad 
a_4(\varphi,\psi)\leq \dfrac{1}{\lambda}\Vert \varphi\Vert_{0,\Omega} \Vert \psi\Vert_{0,\Omega},\nonumber\\
 b_3(q,\psi)\leq \frac{\alpha}{\lambda} \Vert q\Vert_{0,\Omega} \Vert \psi\Vert_{0,\Omega},\qquad 
a_5(p,q)\leq \left( c_0+\dfrac{\alpha^2}{\lambda}\right)\Vert p\Vert_{0,\Omega} \Vert q\Vert_{0,\Omega},
\end{gather}
for all $\bu,\bgamma \in \bH^1(\Omega)$, $\bv,\bzeta \in \bH(\vdiv,\Omega)$,   $\bomega,\btheta \in \bH(\bcurl,\Omega)$, $\varphi,\psi,p,q  \in \rL^2(\Omega)$.

\begin{lemma}\label{cont-quasicoerc}
Consider the bilinear form defined in \eqref{bilinearform}.
For all $\vec{\boldsymbol{x}} \in \bX$, there exists $\vec{\boldsymbol{y}}\in \bX$ such that
\begin{equation} \label{BNB:quasi-inf-sup}
  \langle \mathcal{A}_\epsilon(\vec{\boldsymbol{x}}),\vec{\boldsymbol{y}} \rangle 
  \gtrsim \|\vec{\boldsymbol{x}}\|^2_{\bX}\quad \text{ and } \quad  \|\vec{\boldsymbol{y}} \|_{\bX}\leq \sqrt{2} \|\vec{\boldsymbol{x}}\|_{\bX}.
\end{equation}
\end{lemma}
\begin{proof}
Using first and second Young's inequalities it is not difficult  to assert  that
\begin{equation}\label{young1}
(\bzeta,\bcurl \btheta)_{0,\Omega}\leq  {\dfrac{1}{2}} \Vert \bzeta \Vert_{0,\Omega}^2+{\dfrac{1}{2}}\Vert \bcurl \btheta \Vert_{0,\Omega}^2 \quad \text{ and }\quad (q,\psi)_{0,\Omega}\leq \dfrac{\varepsilon}{2}\Vert q\Vert_{0,\Omega}^2+\dfrac{1}{2\varepsilon}\Vert \psi \Vert_{0,\Omega}^2.
\end{equation}
Next, for a given $\vec{\boldsymbol{z}}:=(\bgamma,\bzeta,\btheta,\psi,q)$ we can construct 
$\vec{\boldsymbol{y}}:=\bigl(\bgamma,\bzeta+\frac{1}{2}\sqrt{\kappa \nu} \bcurl \btheta,-\btheta,-\psi,-q \bigr)$. 
We then invoke \eqref{young1}, so that we can ensure that
\begin{align*}
\langle \mathcal{A}_{\epsilon}(\vec{\boldsymbol{z}}),\vec{\boldsymbol{y}}\rangle&=2\mu |\bgamma|_{1,\Omega}^2+\dfrac{1}{\kappa}\Vert \bzeta \Vert_{0,\Omega}^2+\dfrac{\sqrt{\nu}}{2\sqrt{\kappa}}(\bzeta,\bcurl \btheta)_{0,\Omega}+\dfrac{\nu}{\kappa} \Vert \vdiv \bzeta \Vert_{0,\Omega}^2+\dfrac{\nu}{2}\Vert \bcurl \btheta \Vert_{0,\Omega}^2\\
&\qquad +\Vert \btheta \Vert_{0,\Omega}^2+\dfrac{1}{\lambda}\Vert \psi \Vert_{0,\Omega}^2-\dfrac{\alpha}{\lambda}(q,\psi)_{0,\Omega}-\dfrac{\alpha}{\lambda}(q,\psi)+\left( c_0+\dfrac{\alpha^2}{\lambda}\right) \Vert q\Vert_{0,\Omega}^2\\
&\geq 2\mu |\bgamma|_{1,\Omega}^2+\dfrac{1}{\kappa}\Vert \bzeta \Vert_{0,\Omega}^2+\dfrac{\sqrt{\nu}}{2\sqrt{\kappa}}(\bzeta,\bcurl \btheta)_{0,\Omega}+\dfrac{\nu}{\kappa} \Vert \vdiv \bzeta \Vert_{0,\Omega}^2+\dfrac{\nu}{2}\Vert \bcurl \btheta \Vert_{0,\Omega}^2+\Vert \btheta \Vert_{0,\Omega}^2\\
&\qquad +\left( c_0+\dfrac{\alpha^2}{\lambda}-\dfrac{\varepsilon \alpha}{2\lambda}\right) \Vert q \Vert_{0,\Omega}^2+\left(\dfrac{1}{\lambda}-\dfrac{ \alpha}{2\lambda\varepsilon}\right)\Vert \psi\Vert_{0,\Omega}^2.
\end{align*}
Now, taking $\varepsilon:=\dfrac{5\alpha}{4}+\dfrac{\lambda c_0}{\alpha}$ we can deduce that
\begin{subequations}
\begin{align}
c_0+\dfrac{\alpha^2}{\lambda}-\dfrac{\varepsilon \alpha}{2\lambda}&
\geq \dfrac{3}{8}\left( c_0+\dfrac{\alpha^2}{\lambda}\right),\label{constant1}\\
\dfrac{1}{\lambda}-\dfrac{ \alpha}{2\lambda\varepsilon}
=\dfrac{\dfrac{1}{4}+\dfrac{\lambda c_0}{\alpha^2}}{\dfrac{5}{4}+\dfrac{\lambda c_0}{\alpha^2}}\dfrac{1}{\lambda}
& \geq \dfrac{1}{5} \dfrac{1}{\lambda}, \label{constant2}
\end{align}\end{subequations}
and using \eqref{constant1} and \eqref{constant2} we readily obtain the bound 
\[
\langle \mathcal{A}_{\epsilon}(\vec{\boldsymbol{z}}),\vec{\boldsymbol{y}}\rangle \geq \dfrac{1}{10}\|\vec\bz\|_{\bX}^2.\]
Finally, the definition of the preliminary $\bX$-norm \eqref{quasi-norm} and triangle inequality yield the estimates 
\begin{align*}
\|\vec\by\|_{\bX}^2&=2\mu |\bgamma|_{1,\Omega}^2+\frac{1}{\kappa}\left\Vert \bzeta+\dfrac{\sqrt{\kappa \nu}}{2}\bcurl \btheta \right\Vert_{0,\Omega}^2+{\dfrac{\nu}{\kappa}} \Vert \vdiv \bzeta \Vert_{0,\Omega}^2+ \nu\Vert \bcurl \btheta \Vert_{0,\Omega}^2+\Vert \btheta \Vert_{0,\Omega}^2
 +\left( c_0+\dfrac{\alpha^2}{\lambda}\right)\Vert q \Vert_{0,\Omega}^2+\dfrac{1}{\lambda}\Vert \psi\Vert_{0,\Omega}^2\\
& \leq 2\mu |\bgamma|_{1,\Omega}^2+\dfrac{2\nu}{\kappa}\Vert \bzeta\Vert_{0,\Omega}^2+{\dfrac{\nu}{\kappa}} \Vert \vdiv \bzeta \Vert_{0,\Omega}^2+ \dfrac{3\nu}{2}\Vert \bcurl \btheta \Vert_{0,\Omega}^2+\Vert \btheta \Vert_{0,\Omega}^2
 +\left( c_0+\dfrac{\alpha^2}{\lambda}\right)\Vert q \Vert_{0,\Omega}^2+\dfrac{1}{\lambda}\Vert \psi\Vert_{0,\Omega}^2\\
& \leq 2\|\vec\bz\|_{\bX}^2,
\end{align*}
which completes the proof.
\end{proof}

Using the Banach--Ne\v{c}as--Babu\v{s}ka result, from \eqref{BNB} and \eqref{BNB:quasi-inf-sup} we immediately conclude that problem \eqref{eq:weak2} has a unique solution $(\bu,\bv,\bomega,\varphi,p)$ in the space $\bX$. However we can note that the bound on $(\bu,\bv,\bomega,\varphi,p)$ in the $\bX$-norm will degenerate with $\varepsilon$ (tending either to zero or to infinity).

As a preliminary result required in the sequel, we recall the following classical inf-sup condition for the bilinear form $\hat{b}_1$ (which coincides with that of the Stokes problem, see, e.g., \cite{girault}).   
\begin{lemma}\label{stokes-infsup}
There exists $\beta_1>0$ such that 
  \begin{equation*}
\sup_{\bgamma\in \bH^1_0(\Omega)\setminus \{\boldsymbol{0}\}} \dfrac{\hat{b}_1(\psi, \bgamma)}{|\bgamma|_{1,\Omega}}= \sup_{\bgamma\in \bH^1_0(\Omega)\setminus \{\boldsymbol{0}\}}\dfrac{-(\psi,\vdiv \bgamma)_{0,\Omega}}{|\bgamma|_{1,\Omega}}\geq \beta_1 \Vert \psi\Vert_{0,\Omega} \qquad \text{ for all  } \psi\in \rL_0^2(\Omega).
\end{equation*}
\end{lemma}

\subsection{Parameter-robust well-posedness}
Before addressing the well-posedness of \eqref{eq:weak2} robustly with respect to $\varepsilon$, we first note that the bilinear forms defining the solution operator suggest to modify the metric in $\bX$ and include the following  particular parameter-weighting of the functional spaces
\begin{align*} 
\bX_{\epsilon} & := {\sqrt{2\mu}}\bH^1_0(\Omega)\times\biggl[{\sqrt{\frac{1}{\kappa}}}\bL^2(\Omega)\cap \sqrt{\frac{\nu}{\kappa}} \bH_0(\vdiv,\Omega)\biggr]\times \biggl[\bL^2(\Omega)\cap \sqrt{\nu} \bH_0(\bcurl,\Omega)\biggr] \\
  &\qquad  \times\biggl[{\sqrt{ \frac{1}{\lambda}}} \rL^2(\Omega)\cap {\sqrt{ \frac{1}{2\mu}}} \rL_0^2(\Omega)\biggr] \times \biggl[ 
 {\sqrt{c_0+\frac{\alpha^2}{\lambda}}}\rL^2(\Omega)\cap \left(\sqrt{\kappa} \rH^1_0(\Omega) + \sqrt{\frac{\kappa}{\nu}} 
 \rL_0^2(\Omega)\right)
 \biggr].
\end{align*}
An important observation is that $\bX_\epsilon$ contains the same vectors $\vec\bx$ that are in $\bX$, but which are bounded in the norm $\|\cdot\|_\epsilon$ to be defined later. 

Note also that, proceeding similarly as in  \cite{baerland20} (see also \cite{mardal11}), we have decomposed the space for fluid pressure as the sum 
\[\rL_0^2(\Omega)={\sqrt{\dfrac{\kappa}{\nu}}}\rL_0^2(\Omega) + 
                  {\sqrt{\kappa}}{\rH^1_0(\Omega)} 
                  \cap\rL_0^2(\Omega),\] 
and have endowed it with the norm $\|\cdot\|_{r}$ defined, thanks to \eqref{eq:sum}, as follows
\begin{equation}\label{def:norm-r}
\Vert q\Vert^2_{r}:=\inf_{\displaystyle{s\in \sqrt{\kappa} \rH^1_0(\Omega)\cap\rL_0^2(\Omega)}}\left\{\left(\dfrac{\kappa}{\nu}+c_0\right) \|q-s\|_{0,\Omega}^2+c_0 \|s\|_{0,\Omega}^2+\Vert \sqrt{\kappa} \nabla s\Vert_{0,\Omega}^2 \right\}.
\end{equation}
With this norm we see that, for example,  the boundedness of the bilinear form $\hat{b}_1(\cdot,\cdot)$ can be written as follows
\begin{align*}
(q, \vdiv \bzeta)_{0,\Omega}&=(q-s+s, \vdiv \bzeta)_{0,\Omega}=-(\nabla s, \bzeta)_{0,\Omega}+(q-s, \vdiv \bzeta)_{0,\Omega}\\
&\leq\left\{ \kappa|s|_{1,\Omega}^2+\dfrac{\kappa}{\nu}\Vert q-s\Vert_{0,\Omega}^2 \right\}^{1/2} \left\{ \dfrac{1}{\kappa}\|\bzeta\|_{0,\Omega}^2+\dfrac{\nu}{\kappa}\Vert \vdiv \bzeta\Vert_{0,\Omega}^2 \right\}^{1/2},
\end{align*}
for all $s \in \rH_0^1(\Omega)$, so from \eqref{def:norm-r} we have  
\[
(q, \vdiv \bzeta)_{0,\Omega}\leq \Vert q\Vert_{r} \left\{ \dfrac{1}{\kappa}\|\bzeta\|_{0,\Omega}^2+\dfrac{\nu}{\kappa}\Vert \vdiv \bzeta\Vert_{0,\Omega}^2 \right\}^{1/2}.
\]
On the other hand, we have the following robust-in-$\epsilon$ inf-sup condition for $\hat{b}_1(\cdot,\cdot)$.
\begin{lemma}\label{inf-sup-lemma}
There exists $\beta_0>0$ independent of the parameters in $\epsilon$, such that 
\begin{equation}\label{inf-sup-ep}
\sup_{\bzeta\in \bH_0(\vdiv, \Omega)\setminus \{\boldsymbol{0}\}} \dfrac{-(q,\vdiv \bzeta)}{\left\{ \dfrac{1}{\kappa}\|\bzeta\|_{0,\Omega}^2+\dfrac{\nu}{\kappa}\Vert \vdiv \bzeta\Vert_{0,\Omega}^2 \right\}^{1/2}}\geq \beta_0 \Vert q\Vert_{r} \qquad \text{ for all  } q\in \rL_0^2(\Omega).
\end{equation}
\end{lemma}
\begin{proof}
Owing to \cite{girault}, we know that there exists $\beta_1>0$ (independent of the parameters) such that 
\[
\Vert \nabla q\Vert_{-1,\Omega}\geq \beta_1 \Vert q \Vert_{0,\Omega} \text{ for all  } q\in \rL_0^2(\Omega).
\]
Then, for the operator $\nabla^{-1}:\nabla \rL_0^2(\Omega)\rightarrow \rL_0^2(\Omega)$ (where $\nabla \rL_0^2(\Omega)$ is a closed subspace of 
$\bH^{-1}(\Omega)$), we can deduce that 
\[
\Vert \nabla^{-1} \Vert_{\mathcal{L}(\nabla \rL_0^2(\Omega),\rL_0^2(\Omega))}\leq \beta_1^{-1}.
\]
Using the Poincar\'e inequality we can find a positive constant $c:=c(\Omega)$ such that 
\[
\Vert q\Vert_{0,\Omega}\leq c|q|_{1,\Omega} \text{ for all }q\in {\rH^1(\Omega)}\cap \rL_0^2(\Omega),
\]
or, equivalently, 
\begin{equation}\label{eq:aux-inf}
\Vert \nabla^{-1} \Vert_{\displaystyle\mathcal{L}(\nabla (\rH^1(\Omega)\cap \rL_0^2(\Omega)),\rH^1(\Omega)\cap \rL_0^2(\Omega))}\leq c.
\end{equation}
Then, we have that 
\begin{equation}\label{aux-03}
\Vert \nabla q \Vert_{\bL^2(\Omega)+\nu^{-1/2}\bH^{-1}(\Omega)}\geq \max\{c,\beta_1^{-1}\}\displaystyle\inf_{{\substack{q=q_1+q_2\\ q_1\in\rH_0^1(\Omega)\cap\rL_0^2(\Omega),\\q_2\in \rL_0^2(\Omega)}}}\left\{ |q_1|_{1,\Omega}^2+\dfrac{1}{\nu}\Vert q_2\Vert_{0,\Omega}^2 \right\}^{1/2}. 
\end{equation}
Multiplying \eqref{aux-03} by {$\sqrt{\kappa}$} and applying algebraic manipulations, we can conclude that  the inf-sup condition \eqref{inf-sup-ep} holds with {$\beta_0:=\max\{c,\beta_1^{-1}\}$}.
\end{proof}


As done in the previous subsection, the unique solvability analysis will also follow from the Banach--Ne\v{c}as--Babu\v{s}ka theory, but now using the space $\bX_{\epsilon}$ endowed with the new norm 
\begin{equation}\label{energynorm}
\|\vec{\bx}\|_{\epsilon}^2:=\|\vec{\bx}\|_{\bX}^2+\|p\|_{r}^2+\dfrac{1}{2\mu}\|\varphi\|_{0,\Omega}^2.
\end{equation}
Problem \eqref{eq:weak2} is written as: Find $\vec{\bx}\in \bX_\epsilon$ such that 
\[ \langle \cA_\epsilon(\vec{\bx}), \vec{\by}\rangle = \cF(\vec{\by}) \qquad \text{for all }\vec{\by}\in \bX_\epsilon,\]
or in operator form as follows
\begin{equation}\label{eq:operator}
\cA_\epsilon(\vec{\bx}) = \cF \qquad \text{in $\bX'_\epsilon$},\end{equation}
and the norm of the solution operator is defined as
\[\|\cA_{\epsilon}\|_{\mathcal{L}(\bX_{\epsilon},\bX_{\epsilon}')} := 
  \sup_{\vec{\bx}, \vec{\by} \in \bX_{\epsilon}\setminus\{\vec\cero\}}\frac{|\langle\cA(\vec{\bx}),\vec{\by}\rangle|}{\|\vec{\bx}\|_{\epsilon}\|\vec{\by}\|_{\epsilon}}.\]

\noindent Translating Theorem \ref{BNB} to the present context, we aim to prove that the operator $\cA_\epsilon$ is continuous, that is  
\begin{equation}
  \label{BNB:cont}
  \langle \cA_\epsilon(\vec{\bx}),\vec{\by} \rangle \lesssim \|\vec{\bx}\|_{\epsilon}\|\vec{\by}\|_{\epsilon} \qquad \forall \vec\bx,\vec\by \in \bX_\epsilon,\end{equation}
and that the following  global inf-sup condition is satisfied
\begin{equation} \label{BNB:inf-sup}
  \sup_{\vec \by \in \bX_\epsilon\setminus\{\vec\cero\}}
  \frac{\langle \cA_\epsilon(\vec{\bx}),\vec{\by} \rangle }
  {\|\vec{\by} \|_{\epsilon}}
  \gtrsim \|\vec{\bx}\|_{\epsilon} 
  \qquad \forall \vec\bx \in \bX_\epsilon.
\end{equation}
\begin{theorem}\label{inf-sup}
Let $\Vert \cdot\Vert_{\epsilon}$ be defined as in \eqref{energynorm}. 
Then the bilinear form induced by $\mathcal{A}_{\epsilon}$ (cf. \eqref{bilinearform}) is continuous and inf-sup stable under the norm 
$\Vert \cdot\Vert_{\epsilon}$, i.e., the conditions \eqref{BNB:cont} and \eqref{BNB:inf-sup} are satisfied. 
\end{theorem}
\begin{proof}
For the continuity of $\mathcal{A}_{\epsilon}$, it suffices to use \eqref{aux-bounds}, the norm definition \eqref{energynorm}, 
 Cauchy--Schwarz inequality,  
the definition of $\mathcal{A}_{\epsilon}$, and \eqref{inf-sup-ep}, to arrive at
\[
\langle \mathcal{A}_{\epsilon}(\vec{\bx} ),\vec{\by}  \rangle \leq 2 \|\vec{\boldsymbol{x}}\|_{\epsilon}\|\vec{\boldsymbol{y}}\|_{\epsilon}.
\]
For the global inf-sup, we take a given $\vec{\bx}\in \bX_\epsilon$, and for lemma \ref{cont-quasicoerc}, there exists $\vec{\by}_1\in \bX$ such that 
\[
\langle \mathcal{A}_\epsilon(\vec{\boldsymbol{x}}),\vec{\boldsymbol{y}}_1 \rangle 
  \geq \dfrac{1}{10} \|\vec{\boldsymbol{x}}\|^2_{\bX}\quad \text{and} \quad   \|\vec{\by}_1 \|_{\bX}\leq \sqrt{2} \|\vec{\bx}\|_{\bX}.
\]
Using Lemma \ref{inf-sup-lemma} and Lemma \ref{stokes-infsup} we can find $\bzeta_2$, $\bgamma_3$ and constants $C_1$, $C_2$, $\hat{C}_1$, $\hat{C}_2$ such that
\[
-(p,\vdiv \bzeta_2)\geq C_1 \Vert p \Vert_r^2 \quad \text{and} \quad \left\{ \dfrac{1}{\kappa}\Vert \bzeta_2 \Vert_{0,\Omega}+\dfrac{\nu}{\kappa}\Vert \vdiv \bzeta_2 \Vert_{0,\Omega}^2 \right\}^{1/2} \leq C_2 \Vert p \Vert_r. 
\]
and
\[
-(\varphi,\vdiv \bgamma)\geq \hat{C}_1 \frac{1}{2\mu}\Vert \varphi \Vert_{0,\Omega}^2 \quad \text{and} \quad \sqrt{2\mu}|\bgamma_3 |_{1,\Omega} \leq \hat{C}_2\frac{1}{\sqrt{2\mu}} \Vert \varphi \Vert_{0,\Omega}.
\]
Taking $\vec{\by}_2:=(\boldsymbol{0},\bzeta_2,\boldsymbol{0},0,0)$, $\vec{\by}_3:=(\bgamma_3,\boldsymbol{0},\boldsymbol{0},0,0)$, {$\delta>0$, and $\hat{\delta}>0$,}  we have that
\begin{align*}
 \langle \mathcal{A}_\epsilon(\vec{\boldsymbol{x}}),10\vec{\boldsymbol{y}}_1+\delta \vec{\by}_2+\hat{\delta} \vec{\by}_3 \rangle 
&\geq  \|\vec{\boldsymbol{x}}\|^2_{\bX}+\delta(a_2(\bv,\bzeta_2)+b_2(\bomega,\bzeta_2)+\hat{b}_1(\bzeta_2,p))+\hat{\delta}(a_1(\bu,\bgamma_3)+b_1(\bgamma_3,\varphi)) \\
&\geq  \|\vec{\boldsymbol{x}}\|^2_{\bX}-\delta \left\{ \dfrac{1}{\kappa}\Vert \bv \Vert_{0,\Omega}+\dfrac{\nu}{\kappa}\Vert \vdiv \bv \Vert_{0,\Omega}^2 \right\}^{1/2} \left\{ \dfrac{1}{\kappa}\Vert \bzeta_2 \Vert_{0,\Omega}+\dfrac{\nu}{\kappa}\Vert \vdiv \bzeta_2 \Vert_{0,\Omega}^2 \right\}^{1/2} \\
& \qquad  -\delta\dfrac{\nu}{\sqrt{\kappa}} \Vert \curl \bomega \Vert_{0,\Omega} \Vert \bzeta_2 \Vert_{0,\Omega}+\delta C_1 \Vert p \Vert_r^2-2\mu \hat{\delta} |\bu|_{1,\Omega}|\bgamma_3|_{1,\Omega}+\hat{\delta}\frac{\hat{C}_1}{2\mu}\Vert \varphi \Vert_{0,\Omega}^2 \\
&  \geq \|\vec{\boldsymbol{x}}\|^2_{\bX}-\dfrac{1}{2}\left(\dfrac{1}{\kappa}\Vert \bv \Vert_{0,\Omega}+\dfrac{\nu}{\kappa}\Vert \vdiv \bv \Vert_{0,\Omega}^2\right)-\dfrac{\nu}{2}\Vert \curl \bomega \Vert_{0,\Omega}^2 -\mu |\bu|_{1,\Omega}^2 +\delta C_1 \Vert p\Vert_r^2 \\
& \qquad +\hat{\delta}\frac{\hat{C}_1}{2\mu}\Vert \varphi \Vert_{0,\Omega}^2-\delta^2\dfrac{1}{\kappa}\Vert \bzeta_2 \Vert_{0,\Omega}-\delta^2\dfrac{\nu}{2\kappa}\Vert \vdiv \bzeta_2 \Vert_{0,\Omega}^2-2\mu \hat{\delta}^2|\bgamma_3|_{1,\Omega}^2 \\
& \geq \dfrac{1}{2} \|\vec{\boldsymbol{x}}\|^2_{\bX}+\delta( C_1-\delta C_2^2) \Vert p\Vert_r^2+\hat{\delta}\frac{1}{2\mu}\left(\hat{C}_1-\hat{\delta}\hat{C_2}^2 \right)\Vert \varphi \Vert_{0,\Omega}^2.
\end{align*}
Then, choosing $\delta:=\dfrac{C_1}{2C_2^2}$ and $\hat{\delta}:=\dfrac{\hat{C}_1}{2\hat{C}_2^2}$, we can deduce the estimates 
\begin{align*}
  \langle \mathcal{A}_\epsilon(\vec{\boldsymbol{x}}),10\vec{\boldsymbol{y}}_1+\delta \vec{\by}_2+\hat{\delta} \vec{\by}_3 \rangle  &\geq \dfrac{1}{2} \|\vec{\boldsymbol{x}}\|^2_{\bX}+\dfrac{C_1^2}{4C_2^2} \Vert p\Vert_r^2+\dfrac{\hat{C}_1^2}{4\hat{C}_2^2}\frac{1}{2\mu}\Vert \varphi \Vert_{0,\Omega}^2\geq \dfrac{1}{2} \min \left\{1,\dfrac{C_1^2}{2C_2^2},\dfrac{\hat{C}_1^2}{2\hat{C}_2^2} \right\}\|\vec{\boldsymbol{x}}\|^2_{\epsilon}, \qquad \text{and}\\
  \Vert 10\vec{\boldsymbol{y}}_1+\delta \vec{\by}_2+\hat{\delta} \vec{\by}_3  \Vert_{\epsilon} &\leq 10\sqrt{2} \Vert \vec{\bx} \Vert_{\epsilon}+\delta C_2 \Vert p \Vert_{r}+\hat{\delta} \hat{C}_2\frac{1}{\sqrt{2\mu}} \Vert \varphi \Vert_{0,\Omega} \leq \max\left\{10 \sqrt{2},\dfrac{C_1}{2C_2},\dfrac{\hat{C}_1}{2\hat{C}_2} \right\}\Vert \vec{\bx} \Vert_{\epsilon}.
\end{align*}
And from these relations we can conclude that:
\[
\sup_{\vec \by \in \bX_\epsilon\setminus\{\vec\cero\}}\frac{\langle \cA_\epsilon(\vec{\bx}),\vec{\by} \rangle }
  {\|\vec{\by} \|_{\epsilon}}\geq \frac{\langle \cA_\epsilon(\vec{\bx}),10\vec{\boldsymbol{y}}_1+\delta \vec{\by}_2+\hat{\delta} \vec{\by}_3 \rangle }
  {\|10\vec{\boldsymbol{y}}_1+\delta \vec{\by}_2+\hat{\delta} \vec{\by}_3 \|_{\epsilon}}\geq \dfrac{1}{2}\dfrac{\min \left\{1,\dfrac{C_1^2}{2C_2^2},\dfrac{\hat{C}_1^2}{2\hat{C}_2^2} \right\}\|\vec{\boldsymbol{x}}\|^2_{\epsilon}}{\max\left\{10 \sqrt{2},\dfrac{C_1}{2C_2},\dfrac{\hat{C}_1}{2\hat{C}_2} \right\}\Vert \vec{\bx} \Vert_{\epsilon}} \gtrsim\Vert \vec{\bx} \Vert_{\epsilon}.
\]

\end{proof}


\subsection{Operator preconditioning}

We recall from, e.g., \cite{kirby10,mardal11}, 
that since $\cA_\epsilon$ maps $\bX_\epsilon$ to its dual, when solving the discrete version of \eqref{eq:operator}, iterative methods are not directly applicable unless a modified problem is considered, for example 
\[ \cB\cA_\epsilon(\vec{\bx}) = \cB \cF \qquad \text{in $\bX_\epsilon$,}\]
where $\cB:\bX_\epsilon'\to \bX_\epsilon$ is an appropriately defined isomorphism. As usual, one can take $\cB$  as the Riesz map (self-adjoint and positive definite) whose inverse defines a scalar product $(\cdot,\cdot)_{\bX_\epsilon}$ on $\bX_\epsilon$, and the operator $\cB\cA_\epsilon$ is also self-adjoint with respect to this inner product. Then  
\[\langle\cA_\epsilon(\vec{\bx}),\by\rangle = (\cB\cA_\epsilon (\vec{\bx}), \vec{\by})_{\bX_\epsilon} \quad \text{and} \quad \| \cB\cA_\epsilon (\vec{\bx})\|_{\bX_\epsilon} =\|\cA_\epsilon (\vec{\bx}) \|_{\bX_\epsilon'},\]
and therefore, using the definition of the operator norms, it is readily deduced that 
\[ \| \cB\cA_\epsilon \|_{\cL(\bX_\epsilon,\bX_\epsilon)}=  \|\cA_\epsilon \|_{\cL(\bX_\epsilon,\bX_\epsilon')} \quad \text{and} \quad 
\| (\cB\cA_\epsilon)^{-1} \|_{\cL(\bX_\epsilon,\bX_\epsilon)}=  \|\cA_\epsilon^{-1} \|_{\cL(\bX_\epsilon,\bX_\epsilon')}.\]
Then, if an appropriate metric is chosen such that the norms of $\cA_\epsilon$ and of $\cA_\epsilon^{-1}$ are bounded by constants independent of the model parameters $\epsilon$, then the condition number of the preconditioned system will also be independent of the model parameters. 

Proceeding then similarly as in, e.g., \cite{lee19,baerland20,boon21}, we consider the following block-diagonal preconditioners focusing on the case of mixed boundary conditions (and therefore not including  contributions related to the zero-mean value of fluid and total pressure) 
\begin{subequations}
\label{eq:preconditioners}
\begin{align}
   \mathcal{B}_1 & = \begin{pmatrix} {(-\bdiv(2\mu \beps))}^{-1}  & 0&0 & 0 & 0\\
    0 &\!\!\bigl({\kappa^{-1}\bI - \frac{\nu}{\kappa} \nabla}\vdiv\bigr)^{-1} & 0 &0 &0 \\ 
    0 & 0 & \!\!(\bI + {\nu}\bcurl)^{-1} & 0 & 0 \\
      0 & 0 & 0 & \!\!\bigl(\bigl(\frac{1}{\lambda}+\frac{1}{2\mu}\bigr) I\bigr)^{-1} & 0 \\
      0 & 0 & 0 & 0 & \!\!\bigl(\bigl(c_0 + \frac{\alpha^2}{\lambda}+{{\kappa}}\bigr) I \bigr)^{-1} \end{pmatrix},\\[2ex]
   \mathcal{B}_2 & = \begin{pmatrix} {(-\bdiv(2\mu \beps))}^{-1}  & 0&0 & 0 & 0\\
    0 &\!\!\bigl({\kappa^{-1}}\bI - \frac{\nu}{\kappa}\nabla\vdiv\bigr)^{-1} & 0 &0 &0 \\ 
    0 & 0 & \!\!(\bI + {\nu}\bcurl)^{-1} & 0 & 0 \\
      0 & 0 & 0 & \!\!\bigl(\bigl(\frac{1}{\lambda}+\frac{1}{2\mu}\bigr) I \bigr)^{-1} & 0 \\
      0 & 0 & 0 & 0 & \!\!\bigl((c_0 + \frac{\alpha^2}{\lambda}
      ) I - {\kappa}\Delta\bigr)^{-1}  \end{pmatrix},\\[2ex]
       \mathcal{B}_3 & = \begin{pmatrix}  {(-\bdiv(2\mu \beps))}^{-1}  & 0&0 & 0 & 0\\
    0 &  \bigl({\kappa^{-1}}\bI -  {\bigl(1+\frac{\nu}{\kappa}\bigr)\nabla}\vdiv\bigr)^{-1}  & 0 &0 &0 \\ 
    0 & 0 &  (\bI + {\nu}\bcurl)^{-1}& 0 & 0 \\
      0 & 0 & 0 &  \multicolumn{2}{c}{\multirow{2}{*}{{$\widehat{\mathcal{B}}$}}} \\
      0 & 0 & 0  &        \end{pmatrix},
\end{align}\end{subequations}
with 
\[ {\widehat{\mathcal{B}} = \begin{pmatrix} \biggl(\fracd{1}{\lambda}+\fracd{1}{2\mu}\biggr) I
      & \fracd{\alpha}{\lambda} I \\  \fracd{\alpha}{\lambda}I & \biggl(1 + c_0 + \fracd{\alpha^2}{\lambda}\biggr) I\end{pmatrix}^{-1}
        +  \begin{pmatrix} \biggl(\fracd{1}{\lambda}+\fracd{1}{2\mu}\biggr) I
        & \fracd{\alpha}{\lambda} I \\  \fracd{\alpha}{\lambda}I & \biggl(c_0 + \fracd{\alpha^2}{\lambda}\biggr) I -{\kappa} \Delta \end{pmatrix}^{-1}}.\]
        

Note that only $\cB_3$ results from the Riesz map corresponding to $\bX_\epsilon$ with the complete norm as in \eqref{energynorm}, while $\cB_1,\cB_2$ are 
approximations of $\cB_3$. In particular, $\cB_1$ simply considers the parameter weighting suggested by the weak formulation \eqref{eq:weak} {(but taking into account the scaling of the vorticity seminorm according to Remark~\ref{rem:scale0})} combined with the Riesz map associated with the natural regularity of that formulation, and $\cB_2$ includes also the sum of spaces leading to the pressure Laplacian forms  which are key in achieving robustness for Darcy-type problems \cite{baerland20}. The full form $\cB_3$ also includes the non-standard Brezzi--Braess type of block $\widehat{\mathcal{B}}$ for total and fluid pressures, which is  needed in perturbed saddle-point problems with penalty as proposed in \cite[Section 4]{boon21} (see also \cite{braess96}).

\section{Analysis of a finite element method}\label{sec:FE}
Let $\cT_h$ denote a family of   tetrahedral meshes (triangular in 2D) on $\Omega$ and denote by $\cE_h$ the set of all facets (edges in 2D) in the mesh. 
By $h_K$ we denote the diameter of the element $K$ and by $h_F$ we denote the length/area of the facet $F$. As usual, by $h$ we denote the maximum of the diameters of elements in $\cT_h$. For all meshes we assume that they are sufficiently regular (there exists a uniform positive constant $\eta_1$ such that each element $K$ is star-shaped with respect to a ball of radius greater than $\eta_1 h_K$). It is also assumed that there exists $\eta_2>0$ such that for each element and every facet $F\in \partial K$, we have that $h_F\geq \eta_2 h_K$, see, e.g., \cite{quarteroni,ern04}. By $\PP_k (\Theta)$ we will denote be the space of polynomials of total degree at most $k$ defined locally on the domain $\Theta$, and denote by $\bP_k(\Theta)$ and $\mathbb{P}_k(\Theta)$ their vector- and tensor-valued counterparts, respectively. 
By  $\cE_h$ we will denote the set of all facets and will distinguish between facets lying on the interior and the two sub-boundaries  $\cE_h = \cE_h^\mathrm{int} \cup \cE_h^\Gamma$. 

For smooth  scalar fields $w$ defined on~$\cT_h$,  the symbol $w^\pm$  denotes the traces of~$\bw$ on~$e$ that are  the extensions from the interior of the two elements $K^+$ and~$K^-$ sharing the facet $e$. The symbols  $\avg{\cdot}$ and $\jump{\cdot }$ denote, respectively, the average and jump operators defined as 
$\avg{w} := \frac12 (w^-+w^+)$, $ \jump{w } :=  (w^- - w^+)$. The element-wise action of a differential operator is denoted with a subindex $h$, for example, $\nabla_h$  will  denote  the broken gradient operator. 

The discrete spaces that we consider herein correspond, for $k\geq 0$, to the {generalised Taylor--Hood element pair ($\mathbf{P}_{k+2}-\mathrm{P}_{k+1}$)} for the displacement / total pressure approximation, the H(div)-conforming Raviart--Thomas elements of degree $k$ (denoted $\mathbf{RT}_k$) for velocity approximation, the H(curl)-conforming N\'ed\'elec elements of the first kind and order $k+1$ (denoted $\mathbf{ND}_{k+1}$) for filtration vorticity (see, e.g., \cite{brezzi91,gatica14} for precise definitions of these families of spaces), and piecewise polynomials of degree $k$ for the approximation of interstitial pressure 
\begin{align}
 \nonumber \bU_h &:= \{\bu_h \in \bH^1_0(\Omega): \bu_h|_K \in \mathbf{P}_{k+2}(K),\quad \forall K\in \cT_h\},\\
  \nonumber  \bV_h &:= \{\bv_h \in \bH_0(\vdiv,\Omega): \bv_h|_K \in \mathbf{RT}_k(K),\quad \forall K\in \cT_h\},\\
 \nonumber   \bW_h &:= \{\bomega_h\in \bH_0(\bcurl,\Omega): \bomega_h|_K \in \mathbf{ND}_{k+1}(K),\quad \forall K\in \cT_h\},\\
\label{eq:fespaces}  {\bM^m_{h}} &:= {\{\bu_h \in \bL^2(\Omega): \bu_h|_K \in \mathbf{P}_m(K),\quad \forall K\in \cT_h\}}, \\
  \nonumber   \rU_h &:= \{u_h \in \rH^1_0(\Omega): u_h|_K \in \mathrm{P}_{k}(K),\quad \forall K\in \cT_h\},\\
 \nonumber   \rZ_h &:= \{\psi_h \in {C^0(\Omega)}: 
 \nonumber   \psi_h|_K \in {\mathrm{P}_{k+1}(K)},\quad \forall K\in \cT_h\}, \\
 \nonumber \rQ_h &:= \{q_h \in \rL^2(\Omega): q_h|_K \in \mathrm{P}_k(K),\quad \forall K\in \cT_h\}.
\end{align}
Note that other combinations of finite element families are feasible as well (as long as appropriate discrete inf-sup conditions are satisfied). {In 2D we will consider the same space $ \rQ_h$ for both total and fluid pressures and we will use the $\mathbf{P}_2-\mathrm{P}_0$ pair for displacement and total pressure approximation.}

The (conforming) finite element scheme  associated with \eqref{eq:weak2} reads: Find 
$$(\bu_h,\bv_h,\bomega_h,\varphi_h,p_h)\in \bX_h  :=  \bU_h\times\bV_h\times \bW_h \times \rZ_h \times \rQ_h,$$ 
such that
\begin{subequations}\label{eq:weak-discrete}
\begin{alignat}{5}
&a_1(\bu_h,\bgamma_h) &                &                       &+\;b_1(\bgamma_h,\varphi_h)&                       &=&\;B(\bgamma_h)&\qquad\forall \bgamma_h\in\bU_h, \\
&                 &a_2(\bv_h,\bzeta_h) &+\;b_2(\bomega_h,\bzeta_h) &                       &+\;\hat{b}_1(\bzeta_h,p_h) &=&\;F(\bzeta_h) &\qquad\forall\bzeta_h\in \bV_h,\\
&                 &b_2(\btheta_h,\bv_h)&-\;a_3(\bomega_h,\btheta_h)&                       &                       &=&\;0         &\qquad\forall \btheta_h\in \bW_h,\\
&b_1(\bu_h,\psi_h)    &                &                       &-\;a_4(\varphi_h,\psi_h)   &+\;b_3(p_h,\psi_h)         &=&\;0         &\qquad\forall \psi_h\in \rZ_h,\\
&                 &\hat{b}_1(\bv_h,q_h)&                       &+\;b_3(q_h,\varphi_h)      &-\;a_5(p_h,q_h)            &=&\; G (q_h)&\qquad \forall q_h\in \rQ_h.
\end{alignat}
\end{subequations}
Similarly as in the continuous case, we define 
\begin{align*}
\bX_{\epsilon,h}  &:=  {2\mu}\bU_h \times \left[ \sqrt{\frac{1}{\kappa}}\bM_h^{k+1}\cap \sqrt{\dfrac{\nu}{\kappa}}\bV_h\right]\times [\bM_{h}^{k+1} \cap \sqrt{\nu}\bW_h]\times \left[{\dfrac{1}{\lambda}} \rQ_{h}\cap {\dfrac{1}{2\mu}} \rZ_h \right] \times \left[ \sqrt{c_0}\rQ_{h} \cap \left(\sqrt{\dfrac{\kappa}{\nu}} \rQ_h+ \sqrt{\kappa}\rU_h \right) \right],
\end{align*}
{and the associated discrete norm is
\begin{align}
\|\vec{\bx}_h\|_{\eps,h}^2& :=  2\mu \|\beps(\bu_h)\|^2_{0,\Omega}+ \frac{1}{\kappa}\|\bv_h\|^2_{0,\Omega} + \frac{\nu}{\kappa}\|\vdiv\bv_h\|^2_{0,\Omega} +\|\bomega_h\|^2_{0,\Omega} +\nu\|\bcurl\bomega_h\|^2_{0,\Omega} + \frac{1}{2\mu}\|\varphi_h\|_{0,\Omega}^2 \nonumber\\
&   \quad + \inf_{s_h \in \rU_h}\biggl[\left(c_0+\frac{\kappa}{\nu} \right)\|p_h - s_h\|_{0,\Omega}^2 +c_0\|s_h\|^2+ \bigl\|\sqrt{ \kappa}\nabla_hs_h\bigr\|^2_{0,\Omega}\biggr] + \frac{1}{\lambda} \| \varphi_h + \alpha p_h\|^2_{0,\Omega} + c_0 \|p_h\|^2_{0,\Omega} . \label{norm-discrete}
\end{align}}

As in the continuous case, now problem \eqref{eq:weak-discrete} is written as: Find $\vec{\bx}_h\in \bX_{\epsilon,h}$ such that 
\[ \langle \cA_\epsilon(\vec{\bx}_h), \vec{\by}_h\rangle = \cF_h(\vec{\by}_h) \qquad \text{for all }\vec{\by}_h\in \bX_{\epsilon,h},\]
or in operator form as follows
\begin{equation}\label{eq:operator-discrete}
\cA_\epsilon(\vec{\bx}_h) = \cF_h \qquad \text{in $\bX'_{\epsilon,h}$}.\end{equation}
\begin{lemma}
There exists $\beta_1>0$ independent of the parameters in $\epsilon$ and $h$, such that 
\begin{equation}\label{inf-sup-ep-discrete}
\sup_{\bzeta_h\in \bV_h\setminus \{\boldsymbol{0}\}} \dfrac{-(q_h,\vdiv \bzeta_h)_{0,\Omega}}{\left\{ \dfrac{1}{\kappa}\|\bzeta_h\|_{0,\Omega}^2+\dfrac{\nu}{\kappa}\Vert \vdiv \bzeta_h\Vert_{0,\Omega}^2 \right\}^{1/2}}\geq \beta_1 \Vert q_h\Vert_{r} \qquad \text{ for all  } q_h\in \rQ_h.
\end{equation}
\end{lemma}
\begin{proof}
{The proof requires to assume that the continuous inf-sup condition holds. Then, similarly to the proof of that result (Lemma \ref{inf-sup-lemma}), the first part (steps until \eqref{eq:aux-inf}) is a consequence of the fact that the spaces $\bV_h$ and $\rQ_h$ satisfy the usual discrete inf-sup condition for the Stokes problem. Then, it suffices to follow the scaling argument in \eqref{aux-03}, also valid at the discrete level, to complete the desired condition.} 
\end{proof}

Analogously to the continuous case, we need the following conditions to be satisfied to guarantee existence and uniqueness to problem (\ref{eq:operator-discrete}):
\begin{subequations}\label{eq:th42}
\begin{align}
  \label{BNB:contd}
  \langle \cA_\epsilon(\vec{\bx}_h),\vec{\by}_h \rangle &\lesssim \|\vec{\bx}_h\|_{\epsilon}\|\vec{\by}_h\|_{\epsilon,h} \qquad \forall \vec\bx_h,\vec\by_h \in \bX_{\epsilon,h},\\
   \label{BNB:inf-supd}
  \sup_{\vec \by_h \in \bX_{\epsilon,h}\setminus\{\vec\cero\}}
  \frac{\langle \cA_\epsilon(\vec{\bx}_h),\vec{\by}_h \rangle }
  {\|\vec{\by}_h \|_{\epsilon,h}}
  &\gtrsim \|\vec{\bx}_h\|_{\epsilon,h} 
  \qquad \forall \vec\bx_h \in \bX_{\epsilon,h}.
\end{align}\end{subequations}

\begin{theorem}\label{inf-supd}
Let $\Vert \cdot\Vert_{\epsilon,h}$ be defined as in \eqref{norm-discrete}. 
Then the bilinear form induced by $\mathcal{A}_{\epsilon}$ (cf. \eqref{bilinearform}) is continuous and inf-sup stable under the norm 
$\Vert \cdot\Vert_{\epsilon,h}$, i.e., the conditions \eqref{BNB:contd} and \eqref{BNB:inf-supd} are satisfied. 
\end{theorem}
\begin{proof}
{We use again that the pairs $(\bU_h,\rZ_h)$ and $(\bV_h,\rQ_h)$ are inf-sup stable spaces for the usual bilinear form in Stokes problem, which ensures  that we can find discrete versions $\vec{\by}_{h,2}$, $\vec{\by}_{h,3}$ of the  tuples $\vec{\by}_2$ and $\vec{\by}_3$, respectively, constructed in Theorem \ref{inf-sup}. We also note that $\bcurl \bW_h \subset \bV_h$, and therefore we can prove the discrete version of Lemma \ref{cont-quasicoerc}.  Then the desired result is a consequence of repeating the arguments used in the proof of Theorem \ref{inf-sup}.}
\end{proof}

\begin{theorem}\label{Cea:estimate}
For given $\boldsymbol{b},\boldsymbol{f}\in \bL^2(\Omega)$ and $g\in L^2(\Omega)$,  problem \eqref{eq:weak-discrete} has a unique solution $(\bu_h,\bv_h,\bomega_h,\varphi_h,q_h) \in \bX_{\epsilon,h}$.
In addition, the solution satisfies the following continuous dependence on data
\[
\Vert (\bu_h,\bv_h,\bomega_h,\varphi_h,p_h)\Vert_{\epsilon}\lesssim (\|\boldsymbol{b}\|_{0,\Omega}+\|\ff\|_{0,\Omega}+\|g\|_{0,\Omega}),
\]
and the following C\'ea estimate
\begin{align*}
& \Vert (\bu-\bu_h,\bv-\bv_h,\bomega-\bomega_h,\varphi-\varphi_h,p-p_h)\Vert_{\epsilon} \leq \bigl(1+\alpha^{-1}\Vert \mathcal{A}_\epsilon \Vert_{\mathcal{L}(\bX_\epsilon,\bX'_\epsilon)}\bigr) \Vert (\bu-\bgamma_h,\bv-\bzeta_h,\bomega-\btheta_h,\varphi-\psi_h,p-q_h)\Vert_{\epsilon},
\end{align*}
for all $(\bgamma_h,\bzeta_h,\btheta_h,\psi_h,q_h)\in \bX_{\epsilon,h} $, where $\alpha$ is the positive constant associated with \eqref{BNB:inf-supd}.
\end{theorem}
\begin{proof}
The existence and uniqueness of the solution is obtained in a similar way to its counterpart in the continuous level. 
For the corresponding C\'ea estimate, we proceed to denote as $\vec{\boldsymbol{x}}:=(\bu,\bv,\bomega,\varphi,p)$,  $\vec{\boldsymbol{x_h}}:=(\bu_h,\bv_h,\bomega_h,\varphi_h,p_h)$ and $\vec{\boldsymbol{y}_h}:=(\bgamma_h,\bzeta_h,\btheta_h,\psi_h,q_h)$.
From \eqref{BNB:inf-supd}, we can infer that there exists a positive constant $\alpha$ independent of the parameters such that 
\begin{equation*}
  \sup_{\vec{\boldsymbol{z}}_h \in \bX_{\epsilon,h}\setminus\{\vec{\boldsymbol{0}}\}}
  \frac{\langle \mathcal{A}_\epsilon(\vec{\boldsymbol{x}}_h),\vec{\boldsymbol{z}}_h \rangle }
  {\|\vec{\boldsymbol{z}}_h \|_{\epsilon}}
  \geq \alpha  \|\vec{\boldsymbol{x}}_h\|_{\epsilon} 
  \qquad \forall \vec{\boldsymbol{x}}_h \in \bX_{\epsilon,h}.
\end{equation*}
Using the error equation, we readily obtain that $\langle \mathcal{A}_\epsilon(\vec{\boldsymbol{x}}),\vec{\boldsymbol{y}}_h \rangle =\langle \mathcal{A}_\epsilon(\vec{\boldsymbol{x}}_h),\vec{\boldsymbol{y}}_h \rangle $. Furthermore, 
since $\vec{\boldsymbol{y}}_h-\vec{\boldsymbol{x}}_h\in \bX_{\epsilon,h}$  we can deduce that
\begin{align*}
\Vert \vec{\boldsymbol{x}}-\vec{\boldsymbol{x}}_h \Vert_{\epsilon}&\leq \Vert \vec{\boldsymbol{x}}-\vec{\boldsymbol{y}}_h \Vert_{\epsilon}+\Vert \vec{\boldsymbol{y}}_h-\vec{\boldsymbol{x}}_h \Vert_{\epsilon}\\
&\leq \Vert \vec{\boldsymbol{x}}-\vec{\boldsymbol{y}}_h \Vert_{\epsilon}+\alpha^{-1}\sup_{\vec{\boldsymbol{z}}_h \in \bX_{\epsilon,h}\setminus\{\vec{\boldsymbol{0}}\}}
  \frac{\langle \mathcal{A}_\epsilon(\vec{\boldsymbol{y}}_h-\vec{\boldsymbol{x}}_h),\vec{\boldsymbol{z}}_h \rangle }
  {\|\vec{\boldsymbol{z}}_h \|_{\epsilon}}\\
  &\leq \Vert \vec{\boldsymbol{x}}-\vec{\boldsymbol{y}}_h \Vert_{\epsilon}+\alpha^{-1}\sup_{\vec{\boldsymbol{z}}_h \in \bX_{\epsilon,h}\setminus\{\vec{\boldsymbol{0}}\}}
  \frac{\langle \mathcal{A}_\epsilon(\vec{\boldsymbol{y}}_h-\vec{\boldsymbol{x}}),\vec{\boldsymbol{z}}_h \rangle }
  {\|\vec{\boldsymbol{z}}_h \|_{\epsilon}}\\
  & \leq \Vert \vec{\boldsymbol{x}}-\vec{\boldsymbol{y}}_h \Vert_{\epsilon}+\alpha^{-1}\sup_{\vec{\boldsymbol{z}}_h \in \bX_{\epsilon,h}\setminus\{\vec{\boldsymbol{0}}\}}
  \frac{\Vert \mathcal{A}_\epsilon \Vert_{\mathcal{L}(\bX_\epsilon,\bX'_\epsilon)}
  \Vert \vec{\boldsymbol{y}}_h-\vec{\boldsymbol{x}}\Vert_{\epsilon} \Vert \vec{\boldsymbol{z}}_h \Vert_{\epsilon}  }
  {\|\vec{\boldsymbol{z}}_h \|_{\epsilon}}\\
  &\leq (1+\alpha^{-1}\Vert \mathcal{A}_\epsilon \Vert_{\mathcal{L}(\bX_\epsilon,\bX'_\epsilon)}) \Vert \vec{\boldsymbol{x}}-\vec{\boldsymbol{y}}_h \Vert_{\epsilon},
\end{align*}
which finishes the proof. 
\end{proof}

{
Let us recall, from, e.g.,  \cite[Chapters 16, 17 and 22]{ern04}, the following approximation properties of the 
finite element subspaces \eqref{eq:fespaces}, which are obtained using the classical interpolation theory.   
Assume that $(\bu,\bv,\bomega,\varphi,p) \in \bH^{1+s}(\Omega)\times {\bH^{s}(\vdiv,\Omega)\times \bH^{s}(\bcurl,\Omega)} \times {H}^{s}(\Omega) \times {H}^{s}(\Omega)$,  for some
$s\in(1/2,k+1]$. Then there exists $C>0$,
independent of $h$, such that
\begin{subequations}
\begin{align}
\|\beps(\bu-\mathcal{I}_h(\bu)) \|_{0,\Omega} &\leq  C h^{s}| \bu|_{1+s,\Omega},\label{Ap1}\\
\Vert p -\Pi_h (p) \Vert_{0,\Omega}
&\leq C h^{s}|p |_{s,\Omega},\label{Ap2}
\\
\|\bv-\mathcal{I}_h^{RT}(\bv)\|_{0,\Omega}&\leq C h^{s}|\bv|_{s,\Omega} , \label{Ap3}
\\
\|\bomega-\mathcal{I}_h^{N}(\bomega)\|_{0,\Omega} &\leq C h^{s}|\bomega|_{s,\Omega},\label{Ap4}
\end{align}
\end{subequations}
where $\Pi_h:\rL_0^2\rightarrow \rZ_h\subset \rQ_h$ is the $L^2$-projection and $\mathcal{I}_h:\bH^1(\Omega)\rightarrow \bU_h$, $\mathcal{I}_h^{RT}:\bH_0(\vdiv,\Omega)\rightarrow \bV_h$, $\mathcal{I}_h^{N}:\bH_0(\bcurl,\Omega)\rightarrow \bW_h$ are the Lagrange, Raviart--Thomas and N\'edelec interpolators respectively.
}
\begin{theorem}\label{th:cv}
Let $(\bu,\bv,\bomega,\varphi,p)\in \bX_\epsilon$ and $(\bu_h,\bv_h,\bomega_h,\varphi_h,p_h)\in \bX_{\epsilon,h}$ be the unique solutions to the continuous and discrete problems \eqref{eq:weak2} and \eqref{eq:weak-discrete}, respectively. 
Assume that $(\bu,\bv,\bomega,\varphi,p) \in \bH^{1+s}(\Omega)\times {\bH^{s}(\vdiv,\Omega)\times \bH^{s}(\bcurl,\Omega)} \times {H}^{s}(\Omega) \times{H}^{s}(\Omega)$, for some $s \in (1/2,{k+1}]$. Then, there
exists $C > 0$, {independent of the mesh size $h$ and of the model parameters $\epsilon$}, such that
\begin{align*}
\Vert (\bu-\bu_h,\bv-\bv_h,\bomega-\bomega_h,\varphi-\varphi_h,p-p_h)\Vert_{\epsilon}\leq Ch^s\Vert (\bu,\bv,\bomega,\varphi,p)\Vert_{s,\epsilon},
\end{align*}
where 
\begin{align*}
\Vert (\bu,\bv,\bomega,\varphi,p)\Vert_{s,\epsilon}^2 := 
& 2\mu|\bu|_{1+s,\Omega}^2+\frac{1}{\kappa}|\bv|_{s,\Omega}^2+\frac{\nu}{\kappa}|\vdiv \bv|_{s,\Omega}^2+|\bomega|_{s,\Omega}^2+\nu|\bcurl \bomega|_{s,\Omega}^2 +\frac{1}{2\mu}|\varphi|_{s,\Omega}^2+\left(c_0+\frac{\kappa}{\nu} \right)|p|_{s,\Omega}^2+\frac{1}{\lambda}|\varphi+\alpha p|_{s,\Omega}^2.
\end{align*}
\end{theorem}
\begin{proof}
This result follows immediately after choosing the tuple 
\[
(\bgamma_h,\bzeta_h,\btheta_h,\psi_h,q_h):=(\mathcal{I}_h(\bu),\mathcal{I}_{h}^{RT}(\bv),\mathcal{I}_h^{N}(\bomega),\Pi_h(\varphi),\Pi_h(p)),
\]
in Theorem \ref{Cea:estimate}, and then using the estimates \eqref{Ap1}-\eqref{Ap4} together with the following commuting properties of the Raviart--Thomas and N\'ed\'elec operators 
\[
\vdiv \mathcal{I}_{h}^{RT}(\bv)=\Pi_h(\vdiv \bv), \qquad \bcurl \mathcal{I}_{h}^{N}(\bomega)=\Pi_h(\bcurl \bomega).
\]
Finally, it suffices to invoke the fact that $\Vert q \Vert_r\leq \left(\frac{\kappa}{\nu}+c_0 \right)\Vert q \Vert_{0,\Omega}$.
\end{proof}

\begin{remark}\label{rem:41}
{The error estimate above holds true also in the limit cases of near incompressibility ($\lambda \to \infty$) and near impermeability ($\kappa \to 0$). This robustness in these cases (along with variations in other parameters) is explored in Section~\ref{sec:cv-rob}. 
We also stress that in the non-viscous regime ($\nu = 0$) we do not have a velocity Laplacian in the fluid momentum equation \eqref{eq:biot-brinkman1}, and no vorticity is required. In that case the formulation in \eqref{eq:weak2} recovers the 4-field Biot formulation from \cite{boon21}. Optimal convergence for the Biot limit is confirmed  numerically in Section~\ref{sec:accuracy}, below.}
  \end{remark}

\section{Numerical verification}\label{sec:results}
The aim of this section is to experimentally validate the theoretical results presented in Section~\ref{sec:FE}. 
We certify by error convergence verification the proposed finite element method, and then we apply the formulation in the simulation of a representative viscous flow in a poroelastic channel. We also evaluate the robustness of the preconditioners in \eqref{eq:preconditioners}. The numerical implementation uses the open-source finite element framework \texttt{Gridap} \cite{badia20}, and is available in the public domain \cite{caraballo23}. For the solution of the linear systems in the accuracy verification tests we employ the sparse direct method MUMPS, while in the preconditioner tests we use the preconditioner MINRES iterative solver; see the corresponding section below for the full details underlying the set up of the preconditioner tests. 

\begin{table*}[!t]
\centering
\caption{Accuracy test in 2D. Error history (errors for each field variable on a sequence of successively refined grids) using non-weighted spaces, convergence rates, and $\ell^\infty$-norm of the loss of mass (the residual of the mass conservation equation at the discrete level, eqn. \eqref{loss}). \label{table01}}
{\begin{tabular*}{\textwidth}{@{\extracolsep\fill}rcccccccccccc@{\extracolsep\fill}}
\toprule
DoF  & $h$ &  $e_{1}(\bu)\!$  &  \texttt{r}  &  $e_{\vdiv}(\bv)\!$  &  \texttt{r}     &  $e_{\bcurl}(\bomega)$  
&   \texttt{r}  &  $e_0(\varphi)$  &  \texttt{r} &  $e_0(p)$  &  \texttt{r} & $\mathrm{loss}_h$ \\  
\midrule
\multicolumn{13}{c}{Discretisation with $k=0$}\\
\midrule
    93 & 0.7071 & 2.36e+00 & $\star$ & 1.98e+00 & $\star$ & 8.43e+00 & $\star$ & 3.10e+00 & $\star$ & 4.77e-01 &$\star$ & 2.66e-15\\
   309 & 0.3536 & 1.04e+00 & 1.18 & 1.36e+00 & 0.53 & 6.74e+00 & 0.32 & 1.75e+00 & 0.83 & 2.59e-01 & 0.88 & 8.88e-15\\
  1125 & 0.1768 & 4.68e-01 & 1.16 & 6.94e-01 & 0.98 & 3.47e+00 & 0.96 & 8.98e-01 & 0.96 & 9.13e-02 & 1.50 & 1.66e-14\\
  4293 & 0.0884 & 2.27e-01 & 1.04 & 3.49e-01 & 0.99 & 1.74e+00 & 0.99 & 4.52e-01 & 0.99 & 3.87e-02 & 1.24 & 3.68e-14\\
 16773 & 0.0442 & 1.13e-01 & 1.01 & 1.74e-01 & 1.00 & 8.73e-01 & 1.00 & 2.26e-01 & 1.00 & 1.84e-02 & 1.08 & 1.20e-13\\
 66309 & 0.0221 & 5.65e-02 & 1.00 & 8.73e-02 & 1.00 & 4.37e-01 & 1.00 & 1.13e-01 & 1.00 & 9.05e-03 & 1.02 & 1.99e-13\\
\midrule
\multicolumn{13}{c}{Discretisation with $k=1$}\\
\midrule
   221 & 0.7071 & 8.54e-01 & $\star$  & 1.27e+00 & $\star$ & 5.40e+00 & $\star$ & 1.26e+00 & $\star$ & 2.04e-01 & $\star$ & 8.95e-15\\
   789 & 0.3536 & 1.90e-01 & 2.17 & 3.01e-01 & 2.07 & 1.47e+00 & 1.87 & 3.92e-01 & 1.69 & 3.72e-02 & 2.46 & 2.30e-14\\
  2981 & 0.1768 & 4.80e-02 & 1.99 & 7.75e-02 & 1.96 & 3.86e-01 & 1.93 & 1.01e-01 & 1.96 & 7.69e-03 & 2.27 & 5.56e-14\\
 11589 & 0.0884 & 1.20e-02 & 2.00 & 1.95e-02 & 1.99 & 9.84e-02 & 1.97 & 2.54e-02 & 1.99 & 1.82e-03 & 2.08 & 1.15e-13\\
 45701 & 0.0442 & 3.00e-03 & 2.00 & 4.89e-03 & 2.00 & 2.48e-02 & 1.99 & 6.36e-03 & 2.00 & 4.50e-04 & 2.02 & 2.40e-13\\
181509 & 0.0221 & 7.50e-04 & 2.00 & 1.22e-03 & 2.00 & 6.23e-03 & 1.99 & 1.59e-03 & 2.00 & 1.12e-04 & 2.00 & 5.00e-13\\
\bottomrule
\end{tabular*}}
\begin{tablenotes}
  The symbol $\star$ denotes that no experimental convergence rate has been computed at the first refinement level.
 \end{tablenotes} 
\end{table*}

\subsection{Accuracy tests}\label{sec:accuracy} Let us consider the unit square domain $\Omega = (0,1)^2$ together with the manufactured solutions  
\begin{gather*}
   \bu(x,y) = \begin{pmatrix} \sin(\pi [x+y])\\\cos(\pi[x^2+y^2])\end{pmatrix}, \qquad p(x,y) = \sin(\pi x+y)\sin(\pi y),\\ 
\bv(x,y) = \begin{pmatrix}
\sin(\pi x)\sin(\pi y)\\
\cos(\pi x)\cos(2\pi y)  \end{pmatrix}, \qquad \omega(x,y) =  {\sqrt{\frac{\nu}{\kappa}}}\mathrm{curl}\, \bv, \qquad 
\varphi(x,y) = -\lambda\vdiv\bu + \alpha p.
\end{gather*}
The model parameters assume the arbitrary values $\nu = 1$, $\lambda = 1$, $\mu = 1$, $\kappa = 1$, $c_0=1$, $\alpha = 1$; 
and the loading and source terms, together with essential boundary conditions are computed from the manufactured solutions above. The mean value of fluid and total pressure is prescribed to coincide with the mean values of the manufactured pressures, which is implemented by means of a real Lagrange multiplier. Sequences of successively refined uniform tetrahedral meshes are used to compute approximate solutions and to generate the error history (error decay with the mesh size $h$ and experimental convergence rates, using  norms in non-weighted spaces for each individual unknown and at each refinement level). We display the error history in Table~\ref{table01}, where the method confirms an asymptotic optimal convergence of $O(h^{k+1})$ for each variable and for both polynomial degrees. We remark that for this test we use continuous and piecewise polynomials of degree $k+2$ for displacement and discontinuous piecewise polynomials of degree $k$ for total pressure. Note that in 2D, the filtration vorticity is a scalar field  $\omega = \sqrt{\frac{\nu}{\kappa}} \mathrm{curl}\,\bv$ and the appropriate functional space is {$\rH^1_0(\Omega)$}. In the discrete setting we then select
\[\mathrm{W}_h : = \{\omega_h \in  {\rH^1_0(\Omega)}: \omega_h|_K \in \mathrm{P}_1(K), \ \forall K\in \cT_h\}.\]
Moreover, to strengthen the numerical evidence, in the last column of the table we report the ($L^2$-projection onto $\rZ_h$ of the) loss of mass
\begin{equation}\label{loss}
  \mathrm{loss}_h : = -(c_0+\frac{\alpha^2}{\lambda})p_h+\frac{\alpha}{\lambda}\varphi_h-\vdiv(\bv_h)-g,
\end{equation}  
for each refinement level, confirming a satisfaction of the mass conservation equation on the order of machine accuracy. This test has been performed using pure Dirichlet boundary conditions.

{In regards to Remark~\ref{rem:41}, we perform exactly the same error history as reported in Table~\ref{table01}, now in the non-viscous regime (setting $\nu=0$). Of course, then vorticity is zero and optimal convergence is attained for all other fields also in this case, as shown in Table~\ref{table01-no-nu}.}

\begin{table*}[!t]
\centering
\caption{{Accuracy test in 2D. Error history for the non-viscous regime with $\nu=0$ (errors for each field variable except vorticity on a sequence of successively refined grids) using non-weighted spaces, convergence rates, and $\ell^\infty$-norm of the loss of mass (the residual of the mass conservation equation at the discrete level, eqn. \eqref{loss}).} \label{table01-no-nu}}
{\begin{tabular*}{\textwidth}{@{\extracolsep\fill}rcccccccccc@{\extracolsep\fill}}
\toprule
DoF  & $h$ &  $e_{1}(\bu)\!$  &  \texttt{r}  &  $e_{\vdiv}(\bv)\!$  &  \texttt{r}     &  $e_0(\varphi)$  &  \texttt{r} &  $e_0(p)$  &  \texttt{r} & $\mathrm{loss}_h$ \\  
\midrule
\multicolumn{11}{c}{Discretisation with $k=0$}\\
\midrule
    93 & 0.7071 & 2.36e+00 & $\star$ & 1.93e+00  & $\star$ & 3.08e+00 & $\star$ & 2.61e-01 & $\star$ & 2.61e-15\\
   309 & 0.3536 & 1.04e+00 & 1.18 & 1.34e+00  & 0.52 & 1.74e+00 & 0.83 & 1.40e-01 & 0.90 & 1.04e-14\\
  1'125 & 0.1768 & 4.68e-01 & 1.15 & 6.91e-01  & 0.96 & 8.97e-01 & 0.95 & 7.15e-02 & 0.97 & 8.17e-14\\
  4'293 & 0.0884 & 2.27e-01 & 1.04 & 3.48e-01  & 0.99 & 4.52e-01 & 0.99 & 3.59e-02 & 0.99 & 3.37e-13\\
 16'773 & 0.0442 & 1.13e-01 & 1.01 & 1.74e-01  & 1.00 & 2.26e-01 & 1.00 & 1.80e-02 & 1.00 & 3.64e-12\\
 66'309 & 0.0221 & 5.65e-02 & 1.00 & 8.73e-02  & 1.00 & 1.13e-01 & 1.00 & 9.00e-03 & 1.00 & 2.43e-11\\
 \midrule
\multicolumn{11}{c}{Discretisation with $k=1$}\\
\midrule
      221 & 0.7071 & 8.54e-01 & $\star$ & 1.26e+00  & $\star$ & 1.25e+00 & $\star$ & 9.43e-02 & $\star$ & 7.75e-15\\
   789 & 0.3536 & 1.90e-01 & 2.17 & 3.01e-01  & 2.06 & 3.90e-01 & 1.68 & 2.54e-02 & 1.89 & 2.37e-14\\
  2'981 & 0.1768 & 4.81e-02 & 1.99 & 7.76e-02  & 1.95 & 1.01e-01 & 1.96 & 6.48e-03 & 1.97 & 5.83e-14\\
 11'589 & 0.0884 & 1.20e-02 & 2.00 & 1.96e-02  & 1.99 & 2.53e-02 & 1.99 & 1.63e-03 & 1.99 & 1.25e-13\\
 45'701 & 0.0442 & 3.00e-03 & 2.00 & 4.91e-03  & 2.00 & 6.35e-03 & 2.00 & 4.07e-04 & 2.00 & 4.19e-13\\
181'509 & 0.0221 & 7.51e-04 & 2.00 & 1.23e-03  & 2.00 & 1.59e-03 & 2.00 & 1.02e-04 & 2.00 & 7.12e-12\\
\bottomrule
\end{tabular*}}
\begin{tablenotes}
  The symbol $\star$ denotes that no experimental convergence rate has been computed at the first refinement level.
 \end{tablenotes} 
\end{table*}

\subsection{{Robustness of the convergence rates}}\label{sec:cv-rob}

\begin{table*}[!t]
  \centering
  \caption{{Accuracy test in 3D. Error history (total error in the weighted norm
           that leads to the definition of  $\cB_3$ on a sequence of successively refined grids).} \label{table:02}}
{\begin{tabular*}{\textwidth}{@{\extracolsep\fill}rccc@{\extracolsep\fill}}
\toprule
DoF  & $h$ &  $e((\bu,\bv,\bomega,\varphi,p))$  &  \texttt{r} \\  
\midrule
\multicolumn{4}{c}{Discretisation with $k=0$}\\
\midrule
4'675 & 0.6124 & 1.02e+01 & $\star$ \\ 
19'939 & 0.3062 & 3.87e+00 & 1.402 \\ 
11'0083 & 0.1531 & 1.18e+00 & 1.718 \\ 
718'339 & 0.0765 & 3.24e-01 & 1.862 \\ 
5'162'755 & 0.0383 & 8.51e-02 & 1.928 \\ 
\midrule
\multicolumn{4}{c}{Discretisation with $k=1$}\\
\midrule
12'766 & 0.6124 & 3.03e+00 & $\star$ \\
  55'514 & 0.3062 & 7.32e-01 & 2.048 \\
 310'450 & 0.1531 & 1.34e-01 & 2.454 \\
2'041'058 & 0.0765 & 2.03e-02 & 2.715 \\
\bottomrule
\end{tabular*}}
\begin{tablenotes}
  The symbol $\star$ denotes that no experimental convergence rate has been computed at the first refinement level.
\end{tablenotes}
\end{table*}

{Next we} conduct a second test of convergence, now in the unit cube domain $\Omega = (0,1)^3$ and with mixed boundary conditions (displacement, normal velocity and tangential vorticity are prescribed on the sides $x=0$, $y=0$, $z=0$ and known normal stress, flux, and pressures are imposed on the remainder of the boundary) considering closed-form solutions to the vorticity-based Biot--Brinkman equations as follows 
\begin{gather*}
   \bu(x,y,z) =\frac{1}{10} \begin{pmatrix} \sin(\pi [x+y+z])\\\cos(\pi[x^2+y^2+z^2]) \\ \sin(\pi [x+y+z])\cos(\pi [x+y+z]) \end{pmatrix}, \qquad p(x,y,z) = \sin(\pi x)\cos(\pi y)\sin(\pi z),\\ 
\bv(x,y,z) = \begin{pmatrix}
\sin^2(\pi x)\sin(\pi y)\sin(2\pi z)\\
\sin(\pi x)\sin^2(\pi y)\sin(2\pi z)\\ 
-[\sin(2\pi x)\sin(\pi y)+\sin(\pi x)\sin(2 \pi y)]\sin^2(\pi z)  
\end{pmatrix}, \qquad \bomega(x,y,z) = {\sqrt{\frac{\nu}{\kappa}}}\bcurl \bv, \qquad 
\varphi(x,y,z) = -\lambda\vdiv\bu + \alpha p.
\end{gather*}
The 3D example uses generalised Taylor--Hood elements for the displacement/total pressure pair, {but differently from \eqref{eq:fespaces}, we take one polynomial degree higher for the approximation of velocity, vorticity and fluid pressure. This to maintain an overall $O(h^{k+2})$ convergence}. This time we consider the parameter values 
 $\nu = 0.1$, $\lambda = 100$, $\mu = 10$, $\kappa = 10^{-3}$, $c_0=0.1$, $\alpha = 0.1$, 
 and in Table~\ref{table:02} we only show the error decay measured in the total weighted norm 
 {that leads to the definition of  $\cB_3$, that is $\|\vec{\bx}-\vec{\bx_h}\|_\eps = \sqrt{(\vec{\bx}-\vec{\bx_h})\cdot\cB_3^{-1}(\vec{\bx}-\vec{\bx_h})}$}. This  shows an asymptotic $h^{k+2}$ order of convergence.  
For sake of illustration, we depict in Figure \ref{fig:ex1} the approximate solutions obtained after 4 levels of uniform refinement.

\begin{figure}[t]
  \begin{center}
    \includegraphics[width=0.32\textwidth]{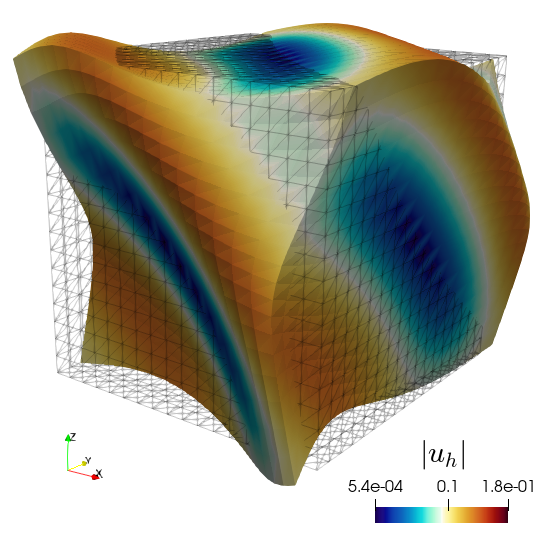}
    \includegraphics[width=0.32\textwidth]{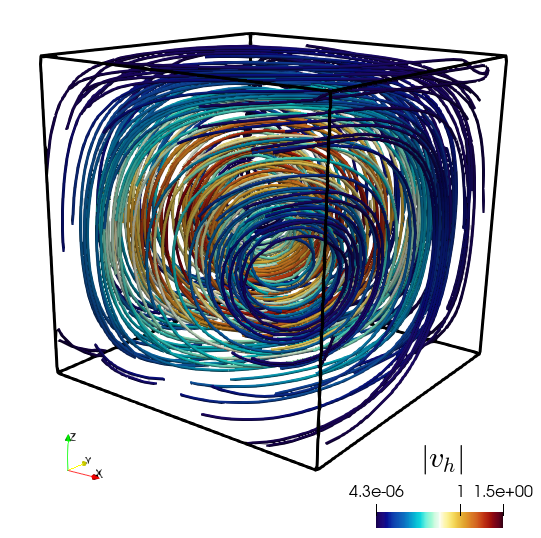}
    \includegraphics[width=0.32\textwidth]{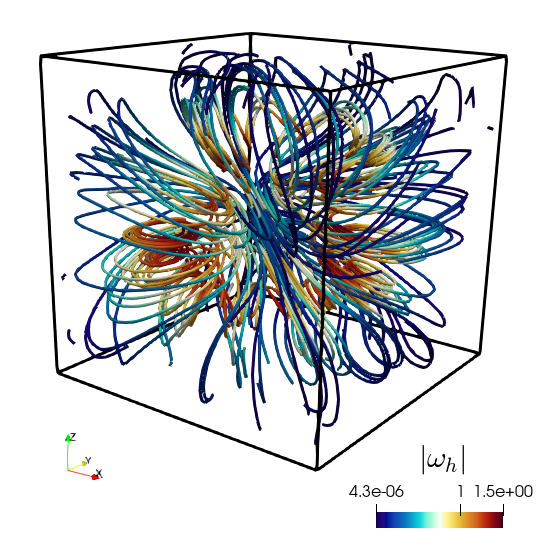}\\
    \includegraphics[width=0.32\textwidth]{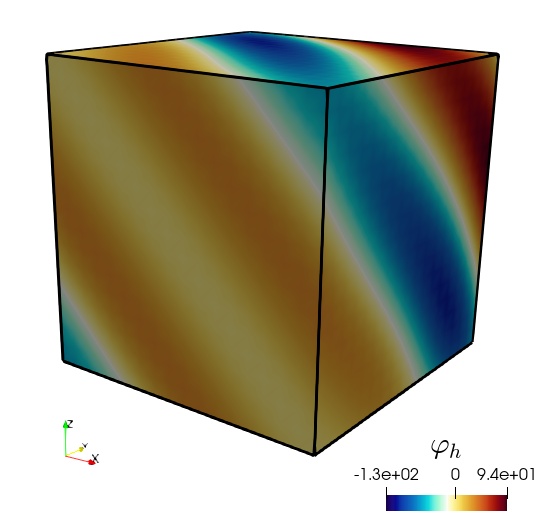}
    \includegraphics[width=0.32\textwidth]{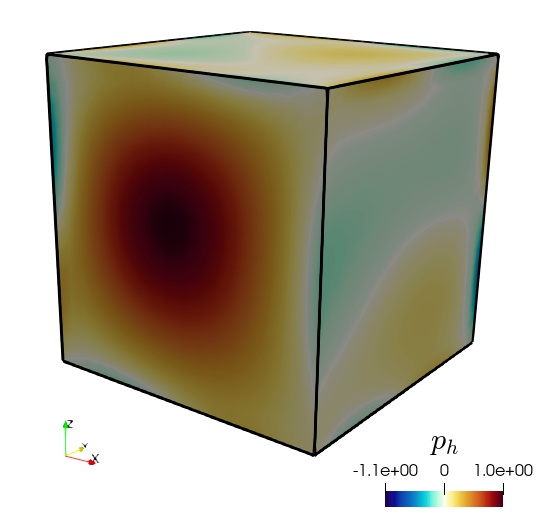}
  \end{center}
  \caption{{Accuracy test in 3D.} Approximate solutions of the Biot--Brinkman equations on a relatively coarse mesh. Displacements on the deformed configuration, velocity streamlines, vorticity vectors, total and fluid pressures computed with the lowest-order method.}\label{fig:ex1}
  \end{figure}

  {We also explore the dependence of the order of convergence upon variation of the base-line parameter values
  used in Table~\ref{table:02}. We consider several sample values of model parameters $\mu \in [1,10^8]$, $\lambda\in [1,10^8]$, $\nu\in [10^{-6}, 1]$, $\kappa\in[10^{-8}, 1]$, $\alpha\in[0,1]$, and $c_0\in [10^{-8},1]$. These ranges are encountered in typical applications of poromechanics of subsurface flows and of linear Biot consolidation of soft tissues \cite{boon21,chen20,lee19,mikelic15,ruiz22}. For the sake of brevity, however, we only report a subset of representative results we obtained with $\mu=1$, $\alpha=1$,  $\lambda=\{1,10^8\}$, $\nu=\{10^{-8}, 1\}$, $\kappa=\{10^{-8}, 1\}$, and $c_0=\{10^{-8},1\}$. In any case, the software in \cite{caraballo23} is written such that the reader might also run additional tests with other parameter values as per-required. We consider the unit cube domain discretised into uniform tetrahedral meshes. In particular, we consider four mesh resolutions, from one up to four levels of uniform refinement of the unit cube discretised with two tetrahedra. We use mixed boundary conditions, with $\Gamma$ being conformed by the faces $x=0$, $y=0$, $z=0$ and $\Sigma$ the remainder of the boundary.}
  
  {The results are shown in Figure \ref{fig:3d_convergence} for $k=0$. Overall, the different convergence curves confirm $O(h^{k+2})$ convergence in the weighted norm
  that leads to the definition of $\cB_3$  regardless of the parameter values, as predicted by the theoretical analysis.}

\begin{figure}[t!]
  \begin{center}
  \includegraphics[width=0.48\textwidth]{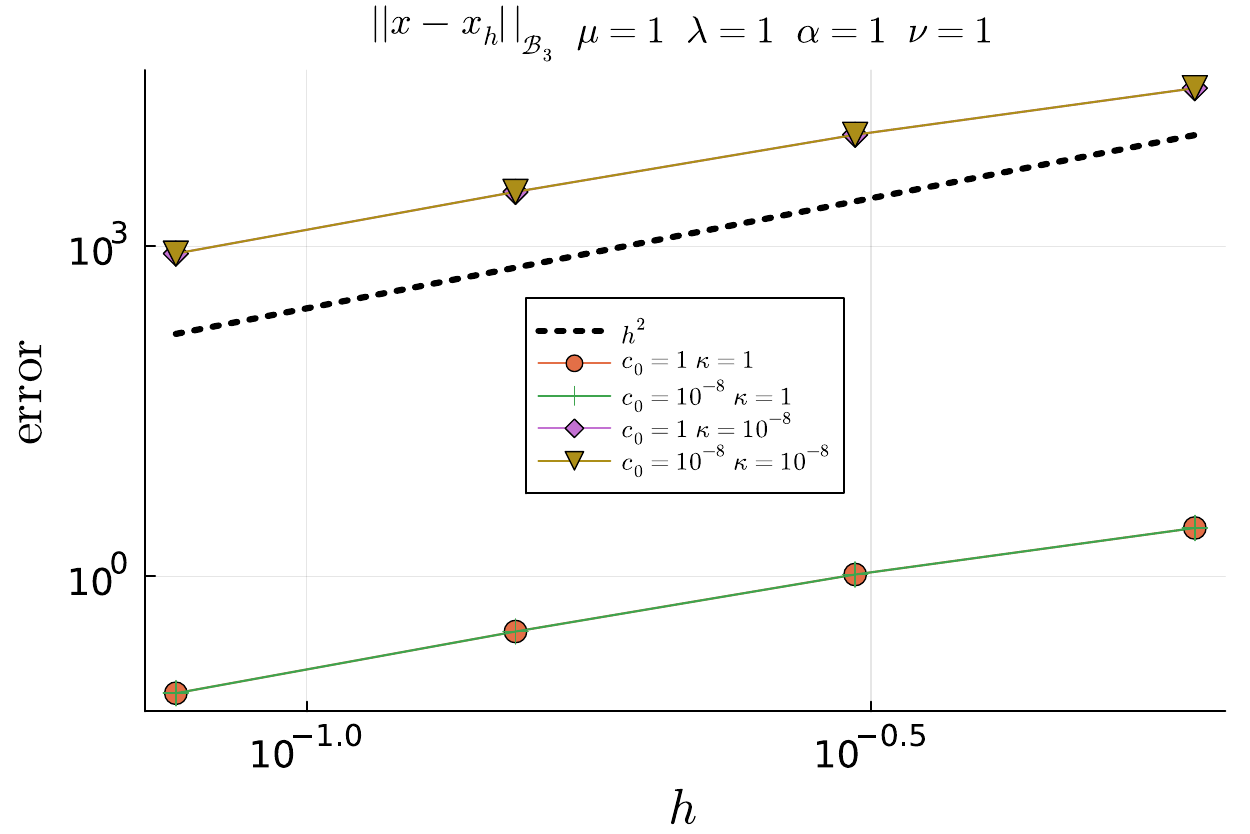}
  \includegraphics[width=0.48\textwidth]{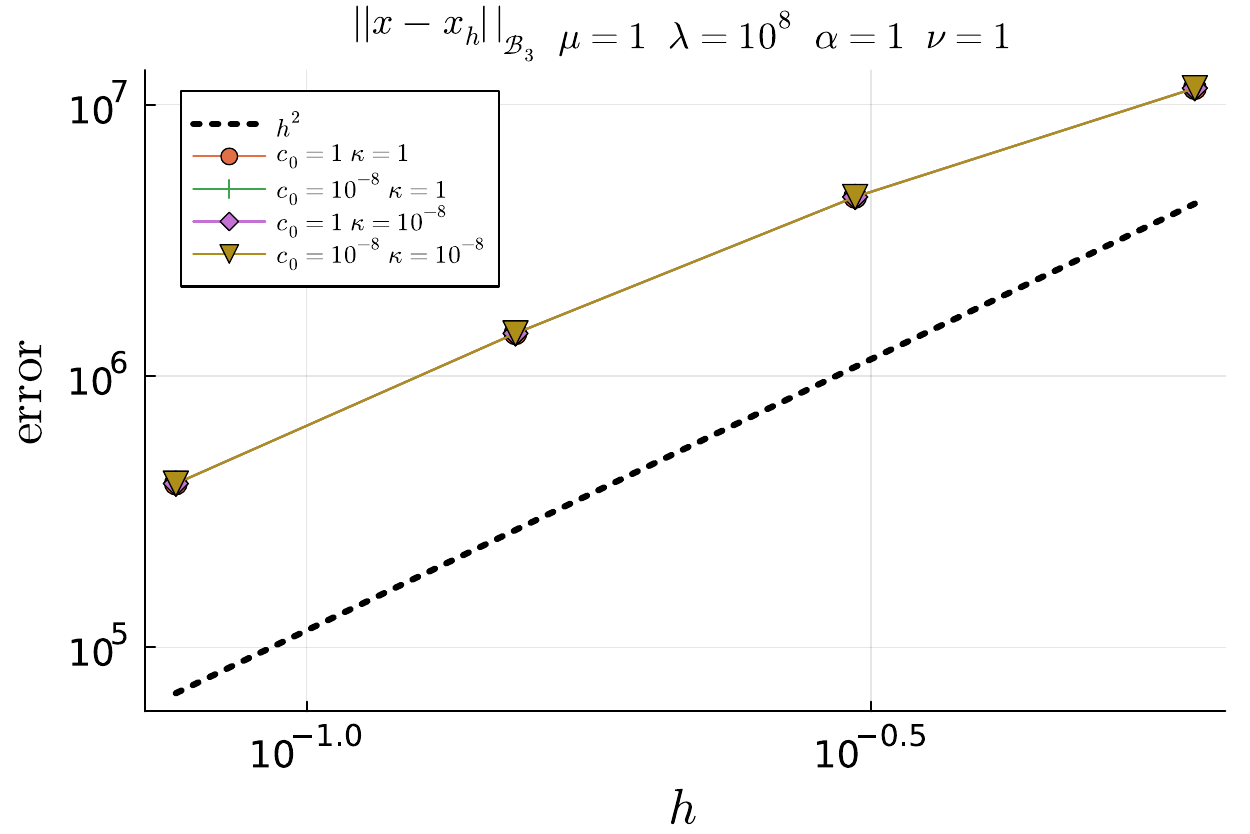} \\
  \includegraphics[width=0.48\textwidth]{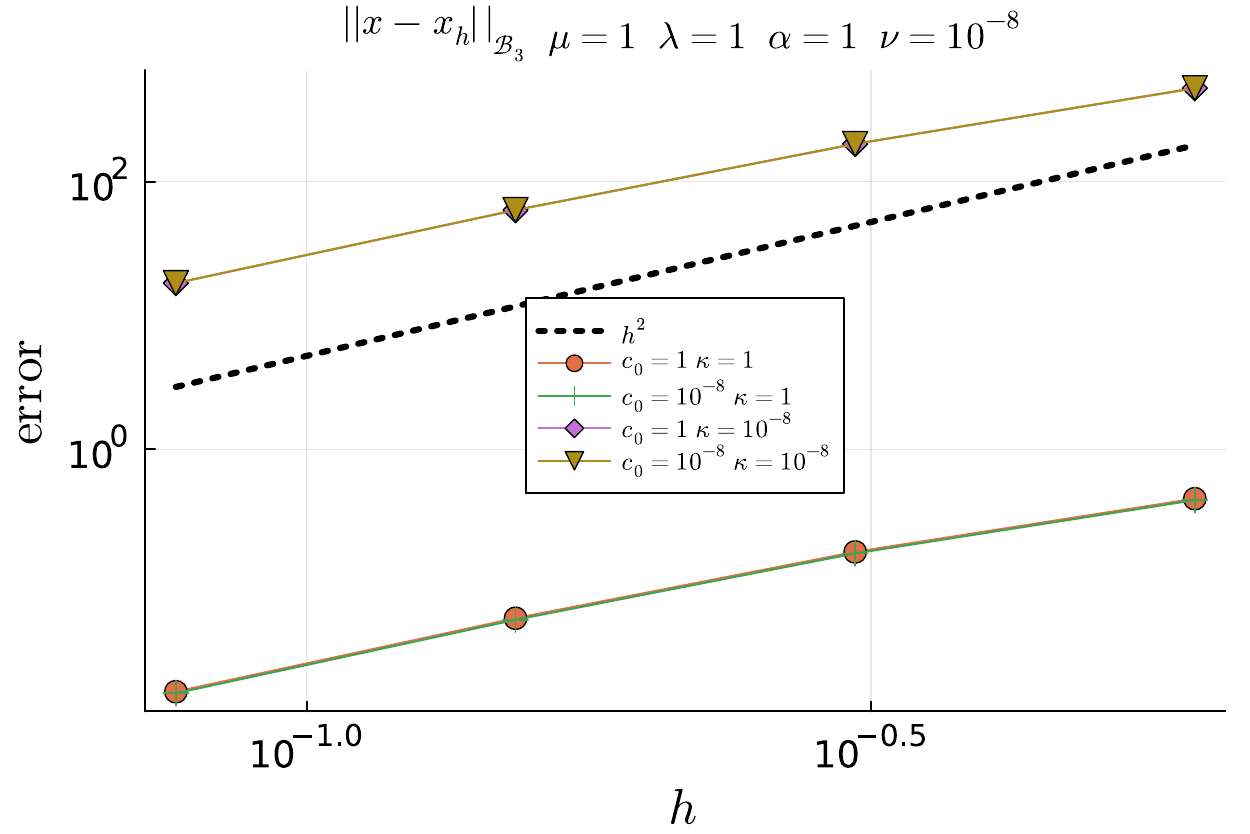}
  \includegraphics[width=0.48\textwidth]{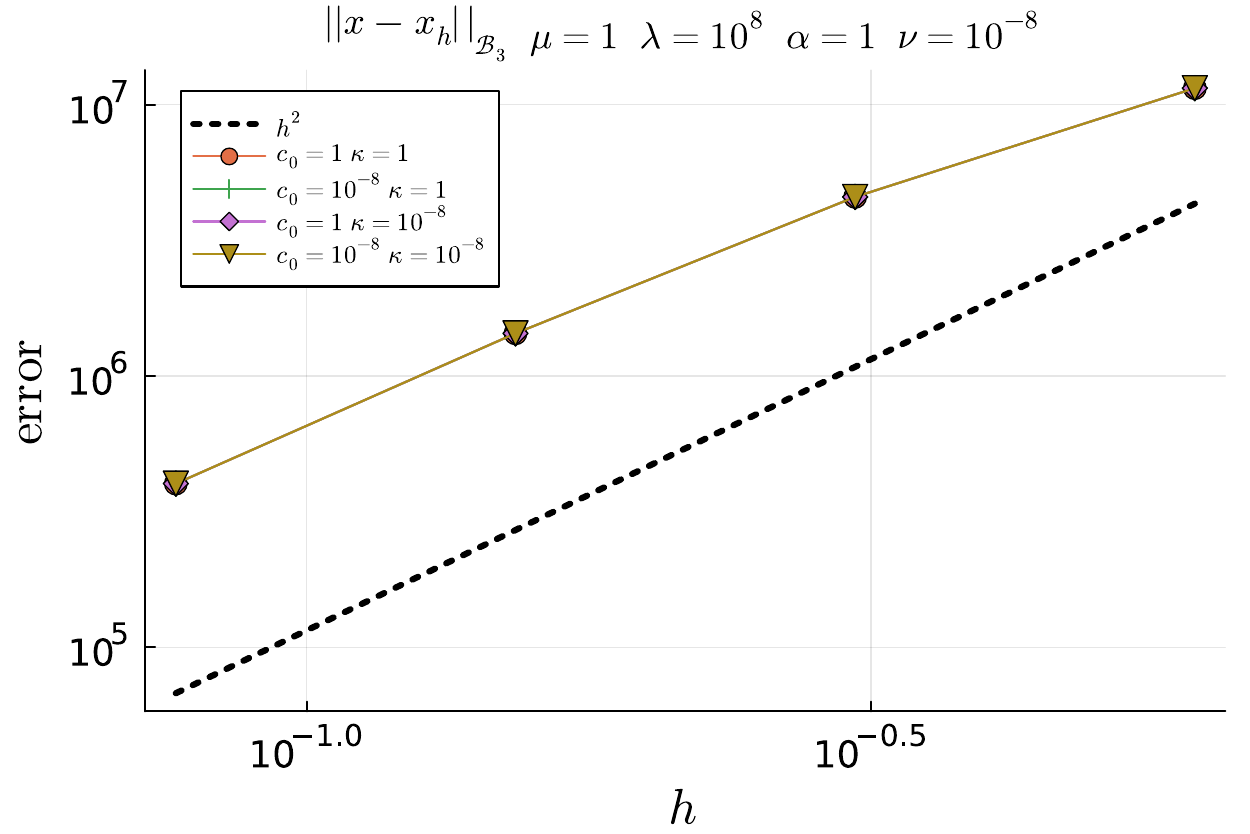}
  \end{center}
  \caption{{Error convergence curves in 3D for $k=0$ and 
                  different combinations of the physical parameters values. 
                  The errors are measured in the weighted norm that leads to the definition of  $\cB_3$.}
                  \label{fig:3d_convergence}}
\end{figure}

\subsection{{Preconditioner} robustness with respect to model parameters}
Finally, we proceed to study the robustness of $\cB_j^{-1}$ in \eqref{eq:preconditioners} with respect to varying model parameters and increasing mesh resolution. {We use the same combination of parameter values as in the previous section, namely $\mu = 1$, $\alpha = 1$,  $\lambda=\{1,10^8\}$, $\nu=\{10^{-8}, 1\}$, $\kappa=\{10^{-8}, 1\}$, and $c_0=\{10^{-8},1\}$, and five levels of uniform refinement of the unit cube discretised with two tetrahedra. We used the same boundary conditions as in the previous section. We used the definition of FE spaces in \eqref{eq:fespaces}
for $k=0$.} The action of the different inverses arising in the diagonal blocks of $\cB_j^{-1}$ in \eqref{eq:preconditioners}
is implemented with the UMFPACK direct solver. For each combination of parameter values and mesh resolution, we used these preconditioners to accelerate the convergence of the MINRES iterative solver. Convergence is claimed whenever the Euclidean norm of the (unpreconditioned) residual of the whole system is reduced by a factor of $10^{6}$, and stopped otherwise if the number of iterations reaches an upper bound of 500 iterations. 
The (discrete) Laplacian operator required for $\cB_2,\cB_3$ acts on discontinuous pore pressure approximations so we use (for a given piecewise-defined field $\eta$) the following form (see \cite{baerland20}) 
\begin{equation}\label{eq:discr-lap} \bigl(-\eta \Delta_h p_h,q_h\bigr)_{0,\Omega} = \sum_{K\in \cT_h} (\eta \nabla_h p_h,\nabla_hp_h)_{0,K} {+ \sum_{e\in\cE^{\mathrm{int}}_h} \bigl\langle \frac{\avg{\eta}}{\avg{h_e}}\jump{p_h},\jump{q_h} \bigr\rangle_{e}} +  \sum_{e\in\cE^{\Gamma}_h} \langle \frac{\eta}{h_e} {p_h},{q_h} \rangle_{e},\end{equation}
while for the lowest-order case we only keep the second and third terms on the right-hand side of \eqref{eq:discr-lap}.

In {Figures~\ref{fig:robustness-a}-\ref{fig:robustness-b}} we show the number of preconditioned MINRES iterations versus number of DoFs for the $\cB_1$, {$\cB_2$}, and $\cB_3$ preconditioners with the particular parameter value combinations mentioned above.
Each of the plots in these figures 
contain 4 curves corresponding each to one of the four combinations of $\kappa=\{10^{-8}, 1\}$, and $c_0=\{10^{-8},1\}$. To facilitate the comparison among $\cB_1$, {$\cB_2$}, and $\cB_3$, the {6} plots in each figure are grouped into two groups of {three} horizontally adjacent plots each. Each group correspond  to a combination of $\lambda=\{1,10^8\}$, $\nu=\{10^{-8}, 1\}$; the particular combination corresponding to a group is indicated in the title of the {three} plots in the group. 

\begin{figure}[t!]
  \begin{center}
  \includegraphics[width=0.75\textwidth]{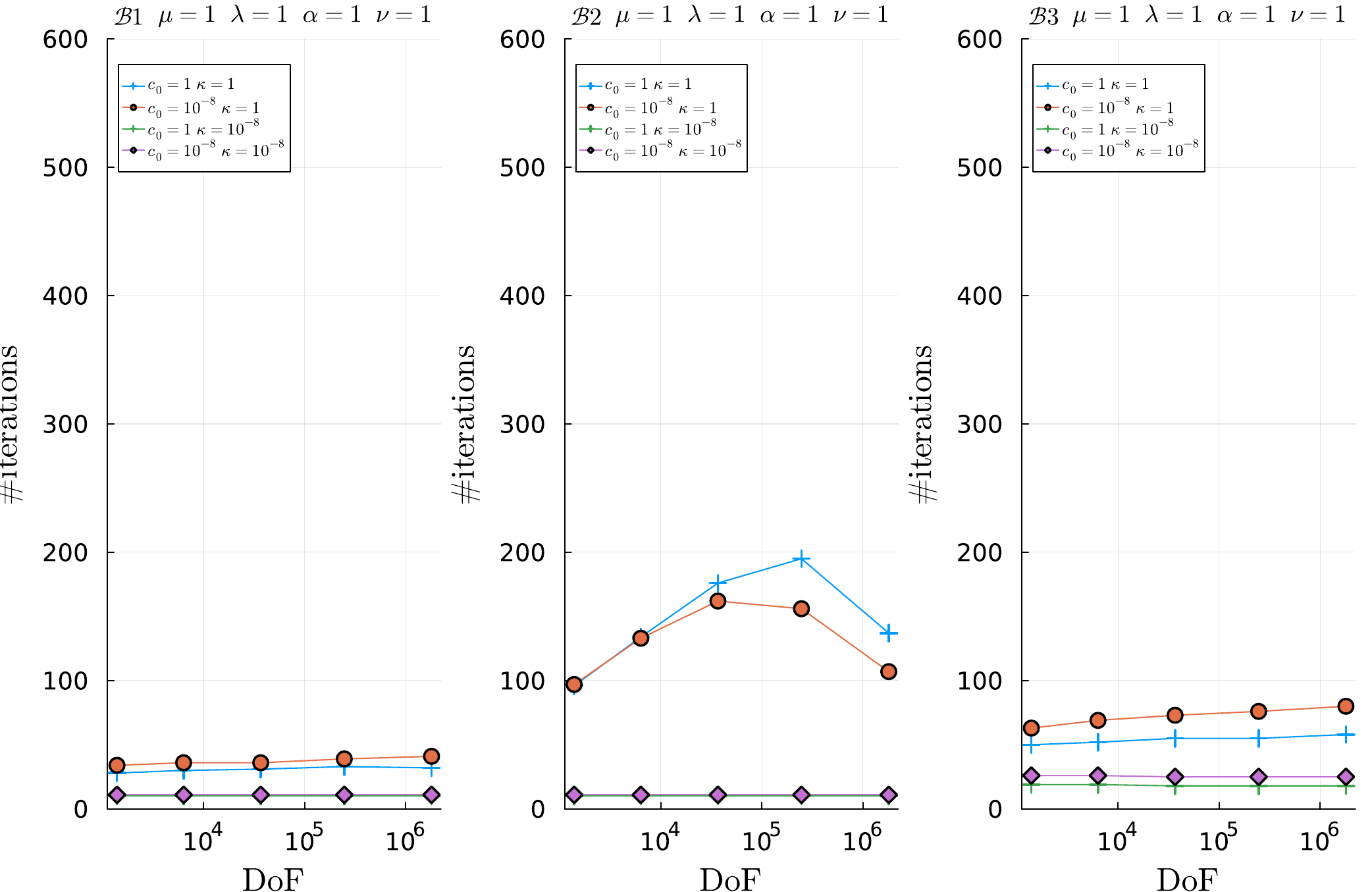}
  \includegraphics[width=0.75\textwidth]{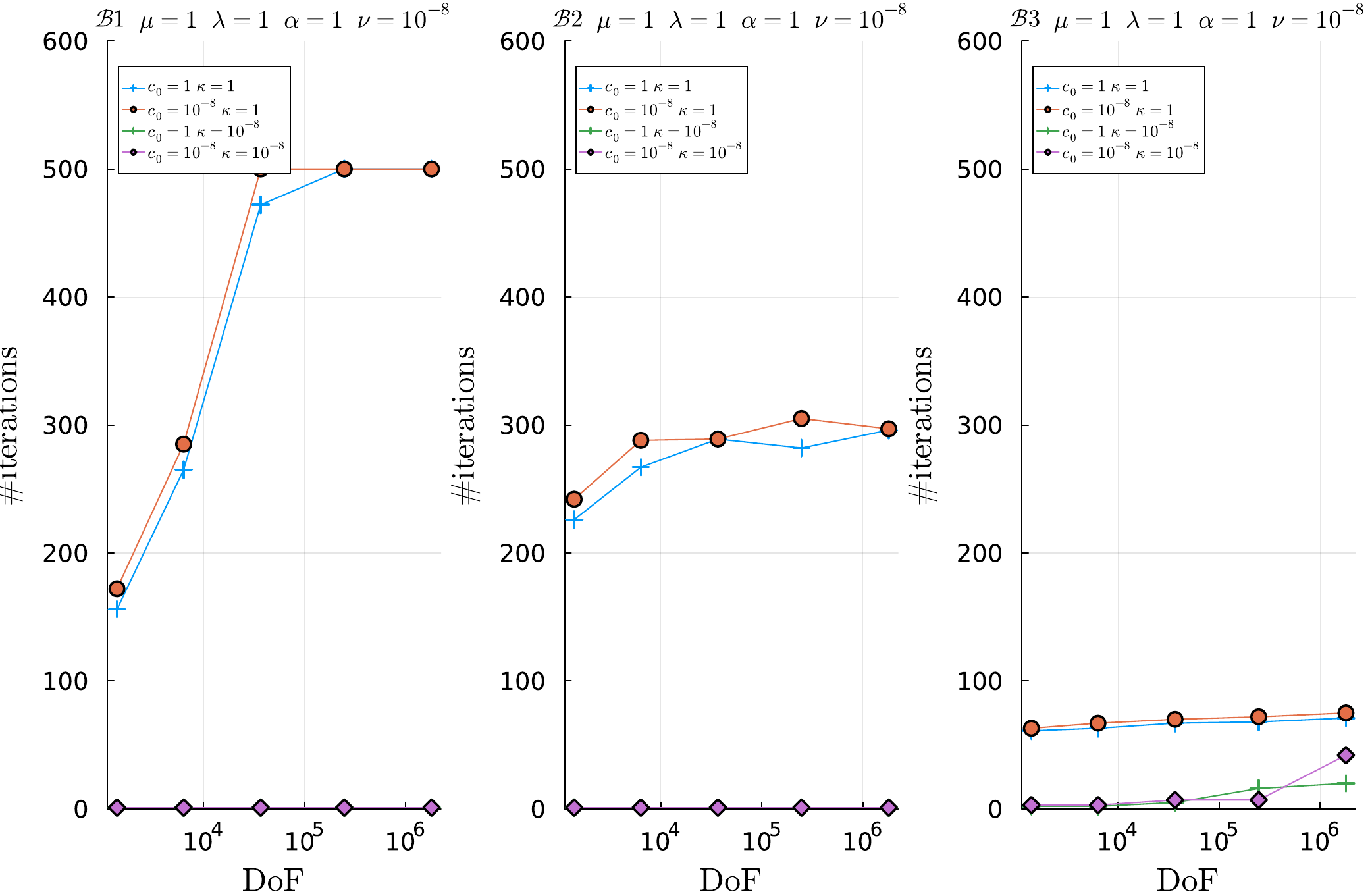}
  \end{center}
  \caption{{Comparison of parameter robustness for preconditioners $\cB_1$, $\cB_2$, and $\cB_3$.}}\label{fig:robustness-a}
  \end{figure}
  
  \begin{figure}[t!]
  \begin{center}
  \includegraphics[width=0.75\textwidth]{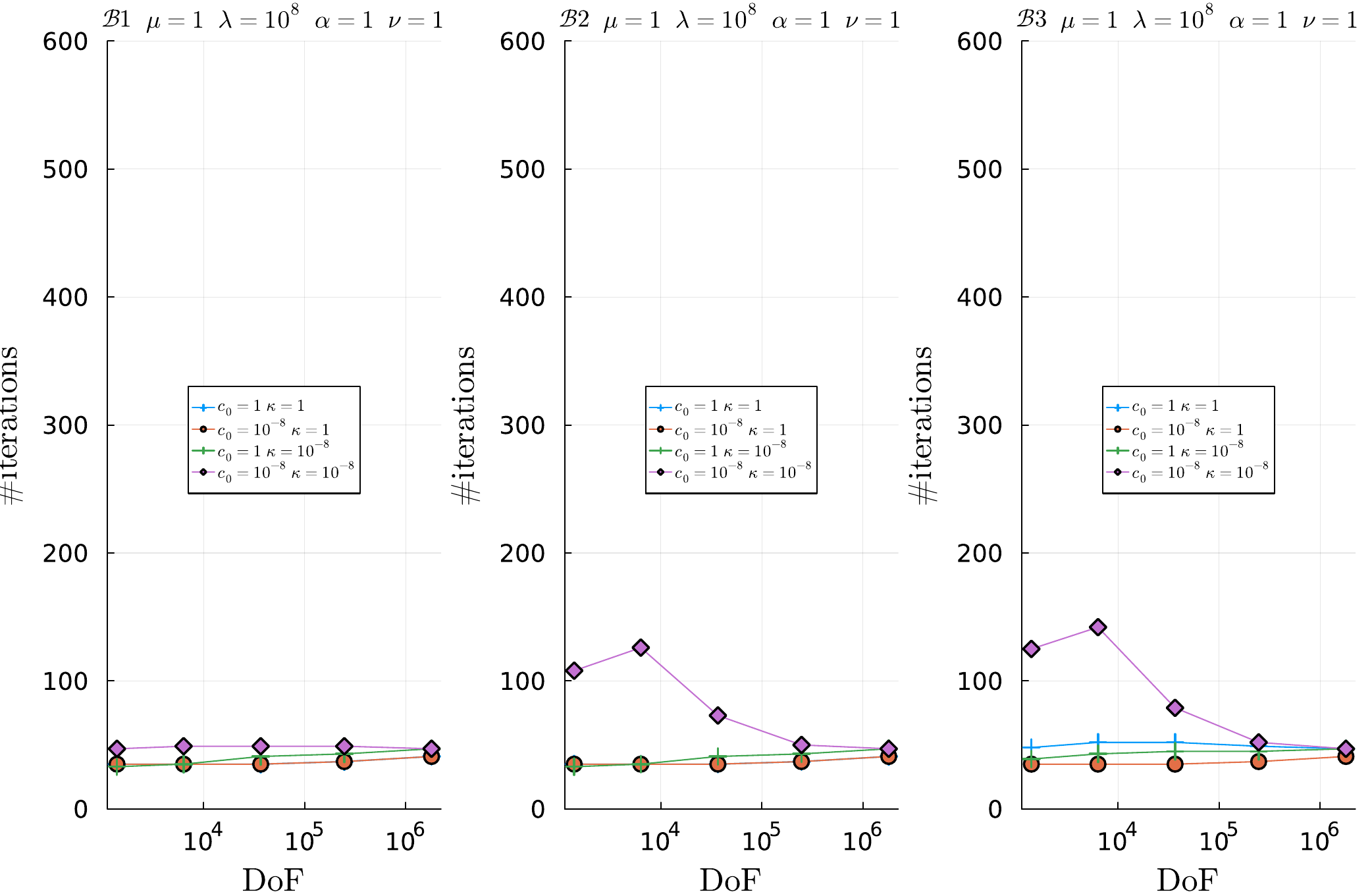}
  \includegraphics[width=0.75\textwidth]{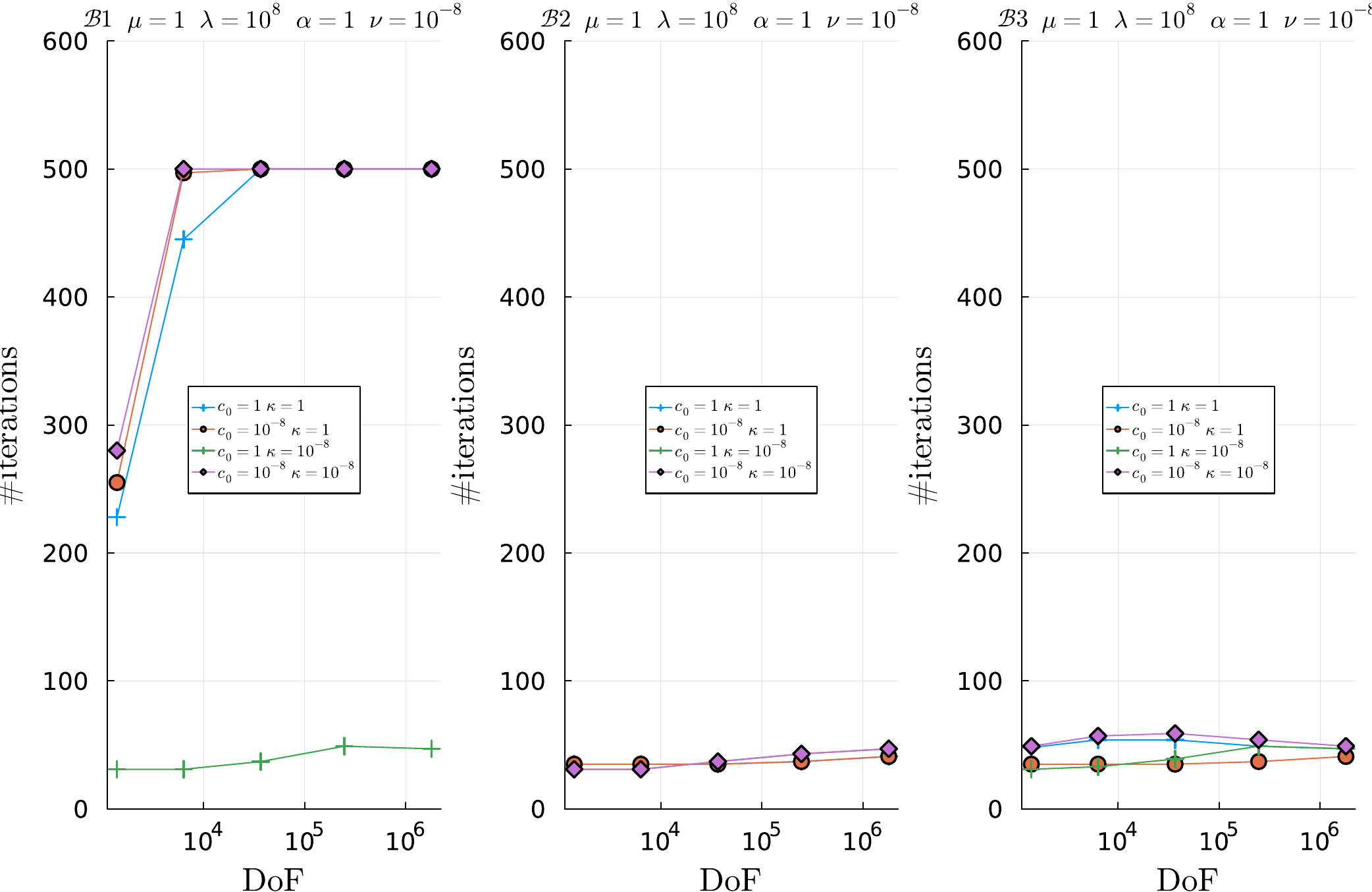}
  \end{center}
  \caption{{Comparison of parameter robustness for preconditioners $\cB_1$, $\cB_2$, and $\cB_3$.}}\label{fig:robustness-b}
  \end{figure}

From {Figures~\ref{fig:robustness-a}-\ref{fig:robustness-b}}, we observe that the number of MINRES iterations with $\cB_3$ reaches an asymptotically constant regime with increasing mesh resolution for all combinations of parameter values tested. This is also the case of $\cB_1$ {and $\cB_2$} in the majority of cases, {except for some combinations of parameter values in which $\nu=10^{-8}$ (see, e.g., $\cB_1$ and $\mu=1$, $\lambda=1$, $\alpha=1$, $\nu=10^{-8}$, $\kappa=1$, and $c_0=\{10^{-8},1\}$)}, where preconditioner efficiency (number of MINRES iterations) significantly degrades (increases) with mesh resolution. {(We recall from Remark~\ref{rem:41} that in the limit $\nu \rightarrow 0$ the system at hand becomes a four-field Biot system~\cite{boon21}.)} For these combinations of parameters, the coupling among the two pressures in the last leading block of $\cB_3$ is essential to retain mesh independence convergence. This observation agrees with the experiments in \cite{boon21} for four-field formulations of Biot poroelasticity and simpler problems (including Herrmann elasticity and reaction-diffusion equation), where comparisons were performed against sub-optimal preconditioners.  
On the other hand, we observe a relatively low sensitivity of the number of MINRES iterations (and thus robustness) with respect to the value of the model of parameters for {all preconditioners} (leaving aside the aforementioned combination of parameter values). For example, 
for $\mu=1$, $\lambda=10^{8}$, $\alpha=1$, $\nu=10^{-8}$, $\cB_3$, and finest mesh resolution, the number of iterations varies between 40 and 50 despite the disparity of scales in the values of $c_0$ and $\kappa$. It is worth noting that 
$\cB_1$ {and $\cB_2$} lead to a similar or even lower number of iterations than $\cB_3$ in most cases, despite it being a computationally cheaper preconditioner. This is the case, {e.g.}, for {and $\cB_1$}, $\mu=1$, $\lambda=1$, $\alpha=1$, $\nu=1$, and all combinations of $\kappa$ and $c_0$ tested.

\section{Concluding remarks}\label{sec:concl}
We have presented a new formulation for the Biot--Brinkman problem using rescaled vorticity and total pressure, and have carried out the stability and solvability analysis of the continuous and discrete problems using parameter-weighted norms. We have derived theoretical error estimates and have subsequently confirmed them numerically; and we have constructed suitable preconditioners that achieve mesh-robustness iteration numbers when varying elastic and porous media flow model parameters. In order to be able to apply these algorithms to more realistic problems one needs to efficiently exploit the vast amount of hardware parallelism available in high-end supercomputers. As future work, we plan to address the parallelisation of the algorithms proposed using the {\tt GridapDistributed.jl} package~\cite{badia22}.

Parts of the theoretical framework advanced herein (in particular, the use of a vorticity-based formulation for the filtration equation) extend naturally to more complex setups, for example to the multiple network generalised Biot--Brinkman model from \cite{hong22}.  Further improvements to the model and to the theory include the rigorous treatment of different types of boundary conditions,  the interfacial coupling with free-flow or with elasticity, and robust a posteriori error estimates.

\bmsection*{Financial disclosure}

None reported.

\bmsection*{Conflict of interest}

The authors declare no potential conflict of interests.


\end{document}